\documentclass[12pt,a4paper]{article}

\title{Complex hyperbolic representations of $SL_2(\R)$ }
\author{Gonzalo Emiliano  Ruiz Stolowicz\thanks{EGG, Institute of Mathematics, EPFL, gonzalo.ruizstolowicz@epfl.ch}}
\usepackage[left=1.5in, right=1in, top=1in, bottom=1in]{geometry}                
\usepackage{amsthm}
\usepackage[utf8]{inputenc}
\usepackage{sectsty}   
\usepackage{amsfonts}
\usepackage{graphicx,import}
\usepackage{amssymb, amsmath}
\usepackage{tikz}
\usepackage{tikz-cd}
\usetikzlibrary{matrix,arrows,decorations.pathmorphing}
\usepackage{enumitem}
\usepackage{mathtools}
\usepackage{dsfont}
\usepackage{amsmath}
\usepackage[capitalise]{cleveref}

\numberwithin{equation}{section}

\newtheorem{teo}{Theorem}[subsection]
\newtheorem{lem}[teo]{Lemma}
\newtheorem{prop}[teo]{Proposition}
\newtheorem{cor}[teo]{Corollary}

\theoremstyle{definition}

\newcommand{\1}{\mathds 1}
\newcommand{\N}{\mathbf N}
\newcommand{\Z}{\mathbf Z}
\newcommand{\R}{\mathbf R}

\newcommand{\C}{\mathbf C}

\newcommand{\F}{\mathbf F}

\newcommand{\hi}{{\mathbf H}^\infty}
\newcommand{\isocuno}{\text{Isom}(\mathbf H_{\mathbf C}^1)_o}
\newcommand{\iso}{\text{Isom}}
\newcommand{\Arg}{{\text{Arg}}}

\newcommand{\Cart}{\text{Cart}}

\begin{document}
	
	\maketitle
	
	\tableofcontents
	
	\begin{abstract}
We construct a two-parameter family of irreducible representations of $SL_2(\R)$ on the infinite-dimensional complex hyperbolic space. To this end, we introduce the notion of \emph{horospherical combination} of two representations. Our family then appears as horospherical combinations of two known one-parameter families.
	\end{abstract}
	\section*{Introduction}
\subsection{Context}
Representations of one semi-simple Lie group into another have no mysteries: by the Karpelevich--Mostow theorem (see \cite{Mostowmostowkarpelevich} or  for a proof in the hyperbolic case see \cite{boubelkarpelevichmost}), they are all ``standard'' in the sense that they correspond to totally geodesic or trivial embedding of the corresponding symmetric spaces of the simple factors. 

\smallskip

The situation changes dramatically when allowing \emph{infinite-dimensional} hyperbolic spaces. The irreducible representations of  $\iso(\mathbf{H}^n_\R)$ into $\iso(\mathbf{H}^\infty_\R)$ have been classified by Monod $\&$ Py in~\cite{monod2014exotic} and it turns out that they form an exotic one-parameter deformation family. The same holds for self-representations of $\iso(\mathbf{H}^\infty_\R)$ (see Monod $\&$ Py  \cite{monodpyselfrepresentations}). Over $\C$, an analogous exotic family has been constructed by Monod in~\cite{monod2018notes} although a complete classification is not yet known.

\bigskip

The present paper focuses on the special case where real and complex hyperbolic spaces meet, namely the Lie group $PSL_2(\R)$. Indeed, this group can be viewed as the connected component of either $\iso(\mathbf{H}^2_\R)$ or $\iso(\mathbf{H}^1_\C)$. This might seem innoccuous since the spaces $\mathbf{H}^2_\R$ and $\mathbf{H}^1_\C$ are isometric (after rescaling in our conventions). However, the the real and complex viewpoint do give us two completely different families of irreducible representations on the infinite-dimensional complex hyperbolic space $\mathbf{H}^\infty_\C$:

On the  one hand from the representations $\iso(\mathbf{H}^2_\R) \to \iso(\mathbf{H}^\infty_\R)$ studied in~\cite{pydelzant} and~\cite{monod2014exotic}, we obtain a family of representations into $\iso(\mathbf{H}^\infty_\C)$ by complexification. Although they are irreducible (in the complex sense), they preserve a totally real subspace and hence have vanishing Cartan invariant.

On the other hand, the representations $\iso(\mathbf{H}^1_\C) \to \iso(\mathbf{H}^\infty_\C)$ constructed in~\cite{monod2018notes} have a non-trivial Cartan invariant, which actually parametrizes that family.

\smallskip

The main result of this paper is an interpolation between these two families of representations of the group $PSL_2(\R) \cong \iso(\mathbf{H}^2_\R)_o \cong \iso(\mathbf{H}^1_\C)_o$. We thus obtain a two-parameter family of representations of $PSL_2(\R)$ into $\iso(\mathbf{H}^\infty_\C)$ which, to the author's knowledge, have not been described before.

\subsection{Statements}
Given an irreducible  representation $\rho$ of $\iso(\mathbf{H}^n_\C)$ (with $1\leq n \leq \infty$) into $\iso(\mathbf{H}^\infty_\C)$, in  \cite{monod2018notes} Monod  has introduced two invariants (with a different notation than here). The first invariant, $\ell(\rho)$ is a \emph{length} invariant already considered for the real case in~\cite{monod2014exotic}. The definition of $\ell(\rho)$ hinges on the fact that for any hyperbolic isometry $g\in \iso(\mathbf{H}^n_\C)$ with displacement length $\ell(g)$, the isometry $\rho(g)$ will also be hyperbolic and its displacement length $\ell(\rho(g))$ is proved to be asymptotically proportional to $\ell(g)$ as $\ell(g) \to\infty$, with a proportionality factor not depending on $g$. This factor is by definition $\ell(\rho)$.

Whilst this first invariant makes sense over the reals as well, the second invariant considered by Monod in \cite{monod2018notes} is specific to the complex case. Recall that the second cohomology of $\iso(\mathbf{H}^n_\C)$ is generated by the K\"ahler class, which can be geometrically realized by the Cartan invariant of triples of points in $\mathbf{H}^n_\C$. Since this class generates the cohomology, the pull-back through $\rho$ of the Cartan invariant in $\mathbf{H}^\infty_\C$ must be some multiple of the Cartan invariant in $\mathbf{H}^n_\C$. For $n=1$, this multiple times $\pi/2$ is by definition the second invariant, the \emph{Cartan argument of $\rho$} denoted $\Arg(\rho)$.

In this language, the first family of representations $PSL_2(\R) \to\iso(\mathbf{H}^\infty_\C)$ from~\cite{monod2014exotic} is parametrized by $0<\ell(\rho)<2$ and satisfies $\Arg(\rho)=0$. The second family, from~\cite{monod2018notes}, is parametrized by $0<\ell(\rho)<1$ but satisfies $\Arg(\rho)=\frac{\pi}{2}\ell(\rho)$. We caution the reader that there is a scaling convention to be taken into account since $\ell$ depends on the metric chosen on $\mathbf{H}^2_\R$ respectively $\mathbf{H}^1_\C$. 

We now have the necessary language to describe our two-parameter family:

\begin{teo}
	For every $0<t<1$ and $0\leq r \leq \frac{t\pi}{2}$, there  exists an irreducible representation $\rho_{t,r}\colon PSL_2(\R) \to \iso(\mathbf{H}^\infty_\C)$ such that $\Arg(\rho_{t,r})=r$ and $\ell(\rho_{t,r})=t$.
	
	Moreover, this representation is unique up to conjugation.
\end{teo}

Our construction is a form of \emph{interpolation} between the families constructed in \cite{monod2018notes} and \cite{monod2014exotic}. In fact, we introduce an operation that we call the \emph{horospherical combination} of two given representations $\rho_1, \rho_2$ that have the same length invariant $\ell(\rho_1) = \ell(\rho_2)$. We denote the resulting representation by $\rho_1\underset{u}{\wedge} \rho_2$, with $0\leq u\leq1$. 

As the name is intended to suggest, this operation resembles a convex combination, but only of the ``horospherical part'' of the representations. For this to make sense, we need to assume that the length invariants coincide. As for the Cartan argument of $\rho_1\underset{u}{\wedge} \rho_2$, it satisfies $$\Arg(\rho\underset{u}{\wedge}\tau)= (1-u)\Arg(\rho)+u\Arg(\tau).$$

	\section{Preliminaries}	
	The  results  described in this section are well known, they are presented here in order to make explicit all the conventions and definitions that will be used.  In the subsections 1.1 and 1.2,  following \cite{burger2005equivariant} and \cite{das2017geometry}, the different hyperbolic spaces are defined and some results about their groups of isometries that will be useful are presented. 
	
	 In section 1.3, following the ideas of \cite{monod2014exotic},   general properties of some complex hyperbolic representations of $SL_2(\R)$ are studied and the notation used in the rest of the article is defined. 
	\subsection{The hyperbolic spaces and their isometries}
	
	Following Burger, Iozzi $\&$ Monod \cite{burger2005equivariant},  
		let $H$ be a separable   Hilbert space over $\mathbf{F}=\R,\C$, with $\dim_\mathbf{F}(H)\in\N_{\geq2}\cup\{\infty\}$,  endowed with a  non-degenerate  form $Q$, linear in the first argument and antilinear in the second. 
	
	Define $$\iota(Q)=\sup\{\dim_\mathbf{F}(W)\mid W\leq H \,\text{and}\, Q|_{W\times W}=0\}$$ 
	and $$\iota_+{Q}=\sup\{\dim_\mathbf{F}(W)\mid W\leq H \,\text{and} \,Q|_{W\times W} \,\text{is positive definite}\}.$$ 
	Suppose that $\iota(Q)=\iota_+(Q)=1.$ A form with these properties will be called a  \textit{form of signature} $(1,n)$, where $n=\dim_\F(H)-1.$
	
	For every $v\in H$, denote $[v]=\mathbf{F} v$. If $dim_\mathbf{F}(H)=m-1$, where $m\geq3$ if $\F=\R$ and $m\geq2$ if $\F=\C$,   define $$\mathbf{H}_\mathbf{F}^m=\{[v]\mid B(v,v)>0\}.$$ 
	The space $\mathbf{H}_\mathbf{F}^m$ is equipped with a metric given by the formula $$cosh(d([v],[w]))=\frac{|B(v,w)|}{B(v,v)^{1/2}B(w,w)^{1/2}}. $$
	
	A $\pm$\textit{-orthogonal decomposition} of $H$ is a $B$-orthogonal decomposition  $$H=W_+\oplus W_-,$$ with $B|_{W_\pm\times W_\pm}$ positive/negative definite.  Given a $\pm$-orthogonal decomposition of $H$, define the sesquilinear form $B_\pm$ as  $B_\pm|_{W_+\times W_+}=B$,  $B_\pm|_{W_-\times W_-}=-B$ and $B(W_+,W_-)=0$. 
	A form of signature $(1,\infty )$ on  $H$  is called \textit{strongly non-degenerate} if for every (any)  $\pm$-orthogonal decomposition,   the space $(H,B_\pm)$ is a Hilbert space (see Lemma 2.4 of \cite{burger2005equivariant}). 
	
	The metric space $( \mathbf{H}_\mathbf{F}^m,d)$ is complete if, and only if, $B$ is strongly non-degenerate (see Proposition 3.3 in \cite{burger2005equivariant}). From now on the space $\mathbf{H}_\mathbf{F}^m$ will be always considered associated to a strongly non-degenerate sesquilinear form and it will be called the $m$-\textit{dimensional}  $\mathbf{F}$-\textit{hyperbolic space} (see Proposition 3.7 of \cite{burger2005equivariant}). 
	For  further reading on these spaces see  \cite{burger2005equivariant},  \cite{das2017geometry} for the infinite-dimensional case,   and  \cite{complexhyperbolic} for the  finite-dimensional complex one. From now on $H$ will denote a separable Hilbert space over $\F$ provided with $B$, a strongly non-degenerate sesquilinear form  from of signature $(1,m)$. 
	
	If $\F=\C$, let $\mathbf{K}=\R,\C$ and if $\F=\R$ define $\mathbf{K}=\R.$ Denote $\pi$ the projectivization map of $H\rightarrow \mathbf{P}(H)$. A $\mathbf{K}-$\textit{hyperbolic subspace} of $\mathbf{H}_\F^m$ is the image under $\pi$ of a closed $\mathbf{K}$-vector subspace $L$ of $H$ such that $B|_{L\times L}$ is non-degenerate of signature $(1,m')$. The restriction of $B$ to $L$ is strongly non-degenerate (see Proposition 2.8 of \cite{burger2005equivariant}), therefore $\pi(L)$ is a (complete) hyperbolic space.  In the complex finite-dimensional case, this is a characterization for totally geodesic subspaces (see 3.1.11 of \cite{complexhyperbolic}). 
	
	For every finite set of points $X$ of $\hi_\mathbf{F}$  there is $W\subset H$, a finite-dimensional space over  $\F$,  that contains representatives of each of  the elements of $X$. The restriction of $B$ to $W$ is a non-degenerate form of signature $(1,n)$, therefore $\pi(W)$ is isometric to a finite-dimensional  $\mathbf{F}$-hyperbolic space. This shows that many statements about  finite sets of points in $\hi_\F$ can be reduced to a finite-dimensional question. 
	For example, the space    $\mathbf{H}_\mathbf{F}^\infty$ is a   geodesically complete  CAT(-1) space because  this is true for every  finite dimensional  $\mathbf{H}_\F^m$ (see Proposition II.10.10 of \cite{bridson2013metric}). 
	
	Every  geodesic ray in $\hi_{\mathbf{F}}$ lies inside a finite dimensional $\mathbf{F}$-hyperbolic space. It is not surprising that $\partial\hi_\F$,   the visual boundary of $\hi_\mathbf{F}$,  is in a natural bijection with the set of isotropic vectors of $H$, because this is true at a finite-dimensional level (see Proposition 3.5.3 in \cite{das2017geometry}). 
	
	Observe that the space $$\{[v]\in\mathbf{P}(H)\mid B(v,v)\geq0\}$$ can be provided with the subspace topology of the projective space (with the quotient topology) associated to $H$. The hyperbolic spaces are Gromov hyperbolic, therefore $\mathbf{H}_\F^m\cup\partial\mathbf{H}_\F^m$  has a natural topology. In this case both topologies are the same and coincide in $\mathbf{H}_{\mathbf{F}}^m$ with the metric topology (see Proposition 3.5.3 of \cite{das2017geometry}). 
	
	Given a CAT(0) space $X$ and $\sigma $ a geodesic ray representing $\xi\in\partial X$. For every $y\in X$, the limit 
	$$b_{\xi,\sigma(0)}(y)=\lim\limits_{t\to\infty }d(y,\sigma(t))-t$$ exists and defines a continuous and convex function (see Lemma II.8.18 of \cite{bridson2013metric}), 
	$$ X\xrightarrow{b_{\xi,\sigma(0)}}\R $$ called the \textit{Busemann function associated to} $\xi$ \textit{and normalized in} $\sigma(0)$.  
	
If $\sigma$ and $\tau$ are two asymptotic geodesic rays, there exists a constant $C$ such that 
	$b_{\xi,\sigma(0)}-b_{\xi,\tau(0)}=C$ (see Corollary II.8.20 of \cite{bridson2013metric}) . Thus, for every  $\xi\in\partial X$, all the Busemann functions associated to $\xi $ have the same level (resp. sublevel) sets. This subsets are called \textit{horospheres} (resp. \textit{horoballs}) centered at $\xi.$ 
	
	For $\mathbf{H}^m_\F$ there is an explicit description of the Busemann functions. If  $x\in \mathbf{H}^m_\F$ it can be shown at a finite-dimensional level that  every geodesic ray $\sigma$ issuing from $x$ admits a lift to $H$ of the shape $t\mapsto cosh(t)\tilde{x}+sinh(t) u$, where $\tilde{x}$ is a lift of $x$, $B(\tilde{x},\tilde{x})=1=-B(u,u)$  and $B(\tilde{x},u)=0$. 
	If  $y\in \mathbf{H}^m_\F$ and $\tilde{y}$ is a lift of $y$ such that $B(\tilde{y},\tilde{y})=1,$
	then 
	$$\begin{array}{rcl}
		b_{\xi,\sigma(0)}(y)&=&\lim\limits_{t\to\infty}d(y,\sigma(t))-t\\
		&=&\lim\limits_{t\to\infty} arccosh\big(|B(\tilde{y}, cosh(t)\tilde{x}+sinh(t) u)|  \big)-t \\
		&=&\ln(|B(\tilde{y},\tilde{x}+u)|). 
	\end{array}$$
	Observe that   $\xi$ is represented by the isotropic vector $\tilde{x}+u$ (see Proposition 3.5.3 of \cite{das2017geometry}). 
	Let $g$ be an isometry of a metric space $X$. Let  $d_g:X\rightarrow\R$ be the function given by $d_g(x)=d(gx,x)$. The \textit{displacement} of $g$ is defined as $$\ell(g)=\inf_{x\in X}\{d_g(x)\}.$$ 
	The following proposition is well known, here just a sketch of the proof is presented. 
	\begin{prop}\label{tricotomia}
		For  every isometry $g$ of a complete CAT(-1) space $X$, there is a trichotomy:
		\begin{enumerate}
			\item The map $g$ is of elliptic type:  $g$ fixes a point in $X$.
			\item The map $g$ is of hyperbolic type: $g$ preserves a geodesic and it does not fix any point in $X$.
			\item The map $g$ is of parabolic type: $g$ fixes a unique point $\xi\in\partial X$,  it does not fix any point in $X$, it  leaves invariant all the horospheres centered at $\xi$ and $\ell(g)=0$.
		
		\end{enumerate}
	\end{prop}
	\begin{proof}
		Suppose $\ell(g)$ is achieved. If $\ell(g)=0$, then 1. holds. If $\ell(g)>0$, then $g$ preserves a geodesic (see Theorem II.6.8 of \cite{bridson2013metric}). The geodesic is unique and there are not fixed points by $g$ in $X$ because the projection onto a non-empty  convex and closed set is a strict contraction (see II.2.12 of \cite{bridson2013metric}). Thus 2. holds. 
		
		Suppose  $\ell(g)$ is not achieved. If $g$ fixes two points in $\partial X$, then it preserves the geodesic connecting them (see Proposition 4.4.4 of \cite{das2017geometry}). Therefore $\ell(g)$ is achieved in this geodesic, but this is a contradiction. Thus $g$ fixes at most one point at infinity. 
		
		For every $n\in\N$ define,  $$H_{n}=\{x\in X\mid d_g(x)\leq \ell(g)+\frac{1}{n}\}.$$ The sets $H_n$ are non-empty,  closed and convex. The family $\{H_{n}\}_{n\in\N}$ is such that  $$I=\bigcap\limits_{n\in\N}H_n=\emptyset.$$ This implies that for every $y\in X$, $d(y,H_{n})\to\infty $ as $n\to\infty $. Therefore
		 $$J=\bigcap\limits_{n\in\N}\partial H_{n}\neq\emptyset$$ (see Theorem 1.1 of \cite{capracelytchak}). Every two points in 
		$\partial X$ are connected by a geodesic (see Proposition 4.4.4 of \cite{das2017geometry}), thus  if  $J$ contains more than one point, then the  geodesic connecting any two elements of the intersection is contained in $I$, which is a contradiction.   The sets $H_{n}$ are $g$-invariant, thus  the only element of $J$  is fixed by $g$. 
		
		Consider a geodesic segment $\sigma$ that represents $\xi$. There exists $c(g)\in\R$ such that f or every $x\in X$, $$\begin{array}{rcl}b_{\xi,\sigma(0)}(gx)-b_{\xi,\sigma(0)}(x)&=&b_{\xi,g^{-1}\sigma(0)}(x)-b_{\xi,\sigma(0)}(x)\\&=&c(g).\end{array}$$
		By the triangle inequality $|c(g)|\leq \ell(g)$, but $\ell(g)=0$ (see Proposition 3.1 of \cite{translationlengthsparabolic}), therefore $b_{\xi,\sigma(0)}$ is $g$-invariant. This completes the proof for 3. 
	\end{proof}
	
	The following proposition  can be deduced from  the arguments in \cite{burger2005equivariant}. 
	\begin{prop}\label{transtividad}
		If  $x\in\mathbf{H}_\F^m$, $y_1,y_2\in\partial \mathbf{H}_\F^m$ and if $s\in O_\F(B)$ is such that $sy_1=y_2 $ and $sy_2=y_1$, then the following hold:
		\begin{enumerate}
			\item The action of $O_\F(B)$ on $\mathbf{H}_\F^m$ is transitive. 
			\item The action of $O_\F(B)_x$ is transitive on metric spheres centered at $x$. 
			\item The action of $O_\F(B)$ is double transitive  on $\partial \mathbf{H}_\F^m$.
			\item If $m<\infty$ and $\F=\R,$ 1.,2., and 3. hold for $SO(1,m).$
			\item $O_\F(B)=O_\F(B)_{y_1}\sqcup \Big(O_\F(B)_{y_1} \cdot s \cdot O_\F(B)_{y_1}\Big).$
		\end{enumerate}
	\end{prop}
	
	The  group $O_\F(B)$ is denoted by $U(1,m)$ (resp. $O(1,m)$) if $\F=\C$ (resp. $\F=\R$).  For every $G\leq O_\F(B)$, denote $PG$ the natural image under projectivization. 
	In fact, 
	$PO(1,m)=\iso(\mathbf{H}_\R^m)$ and $PU(1,m)$ is an index 2 subgroup of $\iso(\mathbf{H}_\C^n)$ (see Theorem 2.2.3 of \cite{das2017geometry}).
	Moreover $$\iso(\mathbf{H}_\C^m)_o=PU(1,m)$$ and  $$\iso(\mathbf{H}_\R^m)_o=PSO(1,m).$$ 
		Observe  that the diagonal matrix act trivially on $\mathbf{H}_\F^m$, 
	therefore if  $m< \infty$, then $PSU(1,m)=PU(1,m)$.
	
	The topology of these groups will be the quotient topology of the projectivization map. This topologies coincide, for $m<\infty$,  with the topology of uniform convergence on compact sets. 
	
	Suppose  $\xi\in\partial \mathbf{H}^m_\F$ and $G< \iso_o(\mathbf{H}^m_\F)_{\xi}$. 
	Let  $b_{\xi,\sigma(0)}$ be a Busemann function centered at $\xi$ and normalized in $\sigma(0)$, for some geodesic ray $\sigma$. The geodesic ray $\sigma$ admits a lift  $$\tilde{\sigma}(t)=cosh(t)\tilde{x}+sinh(t)u,$$ with  $u,\tilde{x}\in H$ such that $B(\tilde{x},\tilde{x})=1=-B(u,u)$ and $B(u,\tilde{x})=0$. For every $g\in G$, there exists $c(g)\in \R$ such that for every $y\in \mathbf{H}_\F^m$, 
	$$b_{\xi,\sigma(0)}(y)=b_{\xi,\sigma(0)}(gy)+ c(g).$$ 
	The map $c: G\rightarrow \R$, called the \textit{Busemann character} associated to $\xi$,  is a continuous homomorphism and does not depend on the choice of $\sigma$.
	
	Observe that if $\tilde{y}$ is a normalized lift of $y$, then 
	$$c(g)=\ln\left(\frac{|B(\tilde{y},\tilde{x}+u|}{|B(\tilde{g}\tilde{y},\tilde{x}+u)|}\right),$$
	where $\tilde{g}$ is any linear representative of the isometry $g$. Thus,  if $$\tilde{g}^{}(\tilde{x}+u)=\theta(\tilde{g})(\tilde{x}+u),$$ with $\theta(\tilde{g})\in\C\setminus\{0\}$,  then $c(g)=\ln(|\theta(\tilde{g})|)$. Therefore  the map $g\mapsto |\theta(\tilde{g})|\in \R_{>0}$ is a continuous homomorphism. 
	\begin{prop}\label{kerneldec}
		If $G<PO_\F(B)_\xi$ and $c:G\rightarrow\R$ is the Busemann character associated to $\xi$, then
		\begin{enumerate}
			\item $\ker(c)=\{T\in G\mid T \,\text{is elliptic or parabolic}\}.$
			\item For every $T\in G$, $\ell(T)=|c(g)|$. 
		\end{enumerate}	
	\end{prop}
	\begin{proof}
	
		1.	Suppose $\xi$ is represented by the isotropic element $y_1.$ Let $T\in G$  and let $\tilde{T}$ be a linear representative of $G$. 
		If  $T$ is hyperbolic, $\tilde{T}$ leaves invariant two isotropic lines with respective representatives $y_1$ and $y_2$. Suppose that $B(y_1,y_2)=1$. Thus,  if $\tilde{T}(y_i)=\theta_i y_i$, then $\theta_1\overline{\theta_2}=1$.  
		
		The point $x$ represented by $\frac{1}{\sqrt{2}}( y_1+y_2)$ belongs to the geodesic connecting $y_1$ and $y_2$ because $2d(T(x),x)=d(T^2(x),x)$. 
		Observe that  $d(T(x),x)=|\ln(|\theta_1|)|. $ 
		This implies that 
		$|\theta_1|\neq 1$, and as it was noticed before, $c(g)=\ln(|\theta_1|) $. Therefore  $T\not\in \ker(c)$. 
		
		If $T$ is parabolic, by \cref{tricotomia},  $c(T)=0$. If $T$ is elliptic then  $T$ fixes pointwise every  geodesic ray representing $\xi$ that starts on a $T$-fixed point in $\mathbf{H}_\F^m$.  Therefore $T$ fixes pointwise the entire geodesic containing any of these geodesic rays. Using the arguments for the hyperbolic case it is possible to conclude that $c(g)=0$. 
		
		The point 2. follows from the arguments of 1 and \cref{tricotomia}. 
	\end{proof}
	
	Let $\{e_1,e_2\}$ be the canonical base of $\C^2$ and define and fix the basis $\{\xi_1,\xi_2\}$,  where 
	$\xi_1=\frac{1}{\sqrt{2}}(e_1+e_2)$ and  $\xi_2=\frac{1}{\sqrt{2}}(e_1-e_2)$. Observe that $B(\xi_1,\xi_1)=0=B(\xi_2,\xi_2)$ and  $B(\xi_1,\xi_2)=1$.
	
	In the basis $\{\xi_1,\xi_2\}$ every element of $U(1,1)_{\xi_1}$ has the form 
	$$\begin{pmatrix}
		\lambda&z\\
		0&\gamma
	\end{pmatrix}$$
	with $\lambda\overline{\gamma}=1$ and $Re(\gamma\overline{z})=0$. Thus  for every $T\in\iso_o(\mathbf{H}^1_\C)_{[\xi_1]}$ there exists 
	$$g(\lambda,b)=\begin{pmatrix}
		\lambda&ib\\0&\lambda^{-1}
	\end{pmatrix}\in SU(1,1)$$ with $\lambda>0$ and $b\in\R$ 
	such that $\pi(g(\lambda,b))=T,$ where $\pi$ is the projectivization map.  
	Define the the  group $$P=\{g(\lambda,b)\}_{\lambda>0, b\in \R}\leq SU(1,1).$$ 
	Observe that for  $g(\lambda,b),g(\gamma,b)\in P,$
	$$g(\lambda,b)\cdot g(\gamma,d)=g(\lambda\gamma,\gamma^{-1}b+\lambda d)$$
	and that $\pi_P:P\rightarrow \iso_o(\mathbf{H}_\C^1)_{[\xi_1]}$ is an isomorphism. A transformation $g(\lambda,b)$ is parabolic if, and only if, $\lambda=1$ and $b\neq 0$. All the nontrivial  maps $g(\lambda,0)$ are hyperbolic. 
	
	If   $s\in SU(1,1)$ is defined by $s(\xi_1)=i\xi_2$ and $s(\xi_2)=i\xi_1$,  then by \cref{transtividad} $$\iso(\mathbf{H}_\C^1)_o =\pi(P)\sqcup\pi(P)\pi(s)\pi(P).$$
	
	Every element of $SU(1,1)$ has the form 
	$$	M(\alpha,\beta)=\begin{pmatrix}
		\alpha&\beta\\\overline{\beta}&\overline{\alpha}
	\end{pmatrix}, $$
	where $|\alpha|^2-|\beta|^2=1.$ 
	The map $SU(1,1)\xrightarrow{\psi}SL_2(\R)$ given by 
	$$\psi(M(\alpha,\beta))=\begin{pmatrix}
		Re(\alpha)+ Im(\beta)&Re(\beta)+Im(\alpha)\\
		Re(\beta)-Im(\alpha)&Re(\alpha)-Im(\beta)
	\end{pmatrix}$$ is an isomorphism. 
	Let $T\in SL_2(\R)$ be $$\frac{1	}{\sqrt{2}}\begin{pmatrix}
		1&-1\\1&1
	\end{pmatrix}$$
	and define the map $SU(1,1)\xmapsto{\Psi}  SL_2(R)$ as 
	$\Psi(A)=T^{-1}\psi(A)T^{}$. 
	The map $\Phi$ is such that 
	$$\psi(g(\lambda,b))=\begin{pmatrix}
		\lambda&b\\
		0&\lambda^{-1}
	\end{pmatrix}$$ 
	and $$\psi(s)=\begin{pmatrix} 0&1\\
		-1&0\end{pmatrix}.$$
	The group $SU(1,1)$ admits a simple description in terms of generators and the  relations between them. The following theorem is a well known fact and it  will play a crucial role in the proof of the main theorem of this article. A proof of it can be found in p.  209 of \cite{sl2rlang}. 
	\begin{teo}\label{relacionesygeneradores}
		Let $F$ be the  free group   generated by the family $\{u(r)\}_{r\in\R\setminus0}$ and an element $w$. For $r\neq 0$, denote  $$s(r)=wu(r^{-1})wu(r)wu(r^{-1}).$$ Consider the relations 
		\begin{enumerate}
			\item $u$ is an additive homomorphism. 
			\item $s$ is a multiplicative homomorphism.
			\item $w^2=s(-1)$
			\item $s(a)u(b)s(a^{-1})=u(ba^2)$, for every $a,b\neq 0 $.
		\end{enumerate}
		If  $G$ is   the quotient of $F$ under these relations then $G$ is isomorphic to $SU(1,1). $
	\end{teo}
	Let $SU(1,1) \xrightarrow{\phi}GL_3(\R)$ be the map defined by 
	$$\phi(M(\alpha,\beta))=\begin{pmatrix}
		-\frac{1}{2}(\beta^2+\overline{\beta}^2-\alpha^2-\overline{\alpha}^2)&\frac{i}{2}(-\beta^2+\overline{\beta}^2-\alpha^2+\overline{\alpha}^2)&i(\overline{\alpha\beta}-\alpha\beta)\\
		-\frac{i}{2}(\beta^2-\overline{\beta}^2-\alpha^2+\overline{\alpha}^2)&\frac{1}{2}(\beta^2+\overline{\beta}^2+\alpha^2+\overline{\alpha}^2)&\overline{\alpha\beta}+\alpha\beta\\
		i(\overline{\alpha}\beta-\alpha\overline{\beta})&\overline{\alpha}\beta+\alpha\overline{\beta}&|\alpha|^2+|\beta|^2
	\end{pmatrix}$$
	and let $$T=\frac{1}{\sqrt{2}}\begin{pmatrix}
		0&0&\sqrt{2}\\
		1&-1&0\\
		1&1&0
	\end{pmatrix}. $$
	Define the map  $SU(1,1)\xrightarrow{\Phi}SO(1,2)$ given by $$\Phi(M(\alpha,\beta))=T^{-1}\phi(M(\alpha,\beta))T.$$ 
	The map $\Phi$ is a  homomorphism and  $\ker(\Phi)=\{Id,-Id\}$. With an  appropriate choice of a basis $\{\xi'_1,\xi'_2,u\}$ of $\R^3$, where $$B(\xi'_i,\xi'_i)=0=B(\xi_i,u)$$ and $$B(\xi'_1,\xi'_2)=1=-B(u,u),$$ the map $\Phi$ is such that, 
	$$\Phi(s)=\begin{pmatrix}
		0&1&0\\
		1&0&0\\
		0&0&-1
	\end{pmatrix}, $$
	$$\Phi(g(1,b))=\begin{pmatrix}
		1&b^2&-\sqrt{b}\\
		0&1&0\\0&-\sqrt b &1
	\end{pmatrix}$$
	and 
	$$\Phi(g(\lambda,0))=\begin{pmatrix}
		\lambda^2&0&0\\
		0&\lambda^{-2}&0\\
		0&0&1
	\end{pmatrix}.$$
Every elliptic transformation is contained in a compact subgroup, therefore its image under $\Phi$ is elliptic too (see Proposition II.2.7 of \cite{bridson2013metric}).	Thus, by \cref{transtividad} and the previous argument, $\Phi$  preserves the type.     

	Consider the following commutative diagram,
	$$\begin{tikzcd}
		SU(1,1)\arrow{r}{\Phi}\arrow{d}{}&SO(1,2)\arrow{d}\\
		\text{Isom}(\mathbf{H}_\C^1)_o\arrow{r}{\overline{\Phi}}&\text{Isom}(\mathbf{H}_\R^2)_o, 
	\end{tikzcd}$$ 
	where the vertical arrows are the projectivization maps and $\overline{\Phi}$ is the induced isomorphism.  
	\begin{lem}\label{semultiplicapordos}
		For every $g\in Isom_o(\mathbf{H}_\C^1)$, $$2\ell(g)=\ell(\overline{\Phi}(g)).$$
	\end{lem}
	\begin{proof}
		The map $\overline{\Phi}$ preserves the type (elliptic, parabolic and hyperbolic), therefore it is enough to prove the claim for hyperbolic elements. Up to conjugation, every hyperbolic element in $\iso(\mathbf{H}^1_\C)_o$  has a representative $g(\lambda,0)$ for some $\lambda>0$. For these elements the claim follows from the arguments in the proof of \cref{kerneldec}. 	 	
	\end{proof}
	\subsection{Isometric  representations}
	In this subsection the generalities of  representations into groups of isometries of hyperbolic spaces  and  the tools used for the statements of the last section are discussed. Most of the results and definitions of this subsection are in \cite{monod2018notes}. here they will be  presented in particular for the case that this article deals with. 
	
	Given a topological group $G$, a homomorphism $G\xrightarrow{\rho}\iso (\mathbf{H}_\mathbf{F}^\infty)$ is called a \textit{representation}. If for every $x\in\mathbf{H}_\mathbf{F}^\infty$, the map $g\mapsto \rho(g)x$ is continuous, then the representation is called \textit{continuous}. In this work  every representation will be considered continuous. 
	
	The representation  $\rho$ is called \textit{non-elementary} if it does not fix any point in $\mathbf{H}_\mathbf{F}^\infty\cup \partial \mathbf{H}_\mathbf{F}^\infty$ and it does not permute two points in $\partial \mathbf{H}_\mathbf{F}^\infty$. The representation $\rho$  is called \textit{irreducible} if it is non-elementary and $\mathbf{H}_\mathbf{F}^\infty$ does not admit a proper $G$-invariant $\F$-hyperbolic subspace. 
	
	The next proposition shows that for non-elementary representations the concept of irreducibility behaves in a very different way than in the linear setting.  The proposition   was proved for the real case in Proposition 4.3 of \cite{burger2005equivariant}. The proof for the complex case works exactly in the same way. 
	\begin{prop}
		If $G\xrightarrow{\rho}\iso (\mathbf{H}_\mathbf{F}^\infty) $ is non-elementary, there exists  a unique $G$-invariant $\F$-hyperbolic subspace $L$ such that for every  $G$-invariant $F$-hyperbolic space $N$, $L\subset N$. 
	\end{prop}
	The space $L$ will be called \textit{the irreducible part} of $\rho$.
	The following  result is similar to Proposition 5.1 of \cite{monod2018notes} or to Proposition 2.1 of \cite{monod2014exotic}. The proof for this case can be mimicked from those of any of the aforementioned propositions. 
	\begin{prop}\label{preservaeltipo}
		A non-elementary representation $\iso(\mathbf{H}_\C^1)_o\xrightarrow{\rho} \iso(\mathbf{H}^\infty_\F)$ preserves the type (elliptic, parabolic and hyperbolic) and $\rho(P)$ fixes a unique point $\xi\in\partial\mathbf{H}^m_\F $. 
	\end{prop}
	
	In \cite{monod2018notes} the author developed a Gelfand-Naimark-Segal type of construction for actions by isometries on complex  hyperbolic spaces. This construction is the key idea behind the main result of this paper. 
	
	Given $x,y,z\in \mathbf{H}_\C^m$, the \textit{Cartan argument} of  $(x,y,z)$ is defined as $$\Cart(x,y,z)=\Arg \big(B(\tilde{x},\tilde{y}),B(\tilde{y},\tilde{z})B(\tilde{z},\tilde{x})\big),$$ 
	where $\tilde{x},\tilde{y}$ and $\tilde{z}$ are any  lifts of $x,y$ and $z$. This definition can be extended for triples of distinct points in $\partial \mathbf{H}_\C^m.$
	
	The map $$ \mathbf{H}_\C^m\times  \mathbf{H}_\C^m\times \mathbf{H}_\C^m\xrightarrow{\Cart}\R$$ is an alternating 2-cocycle and its image is contained in $(-\frac{\pi}{2},\frac{\pi}{2})$. In the complex hyperbolic space this invariant for three points plays a very important role. For further reading see   \cite{burgeriozziboundedcohomology}, \cite{complexhyperbolic} and  \cite{monod2018notes}.
	
	A set $X\subset \mathbf{H}_\C^m$ is contained in a real hyperbolic subspace if, and only if, for every $x,y,z\in X$, $\Cart(x,y,z)=0$ (see Lemma 2.1 in \cite{burgeriozziboundedcohomology}). 
	A set $X\subset \mathbf{H}_\F^m$ is called \textit{total} if there is not a proper and  closed $\F$-vector space that contains the lifts of $X$. 
	
	Following \cite{monod2018notes}, a pair $(\alpha, \beta)$ is called a \textit{ G-invariant kernel of hyperbolic type } defined on a topological group $G$, if 
	$$\alpha:G^3\rightarrow\left(-\frac{\pi}{2},\frac{\pi}{2}\right)$$ is a continuous $G$-invariant (with respect to the diagonal action) alternating 2-cocycle, 
	$$\beta: G\rightarrow \R_{>0}$$ is a continuous  function, symmetric  with respect to the inversion of the group, such that $\beta(e)=1$ and such that 
	the map 
	$$(g,k)\mapsto\beta(g)\beta (k)-e^{-i\alpha(g,k,e)}\beta(g k^{-1})$$ is a kernel of positive type. See chapter II C of \cite{propertyTbekka} for the definition and some properties of the kernels of positive type.  
	
	The next result is Theorem 1.11 of \cite{monod2018notes}. 
	\begin{teo}\label{gns}
		The pair $(\alpha, \beta)$ is a $G$-invariant kernel of hyperbolic type if, and only if, there exist, up to a conjugation by an isometry of $\mathbf{H}_\C^m$, a unique  representation $G\xrightarrow{\rho}\iso(\mathbf{H}^m_\C)_o$  and $p\in\mathbf{H}^m_\C$ such that the orbit of $p$ is total and 
		$$\beta(g)=coshd(\rho(g)p,p)$$ and 
		$$\alpha(g_1,g_2,g_3)=\Cart\big(\rho(g_1)p,\rho(g_2)p,\rho(g_3)p\big).$$
		Moreover $\beta $ and $\alpha $ are continuous if, and only if, $\rho$ is orbitally continuous. 
	\end{teo}
Recall from  the proof of the previous theorem that for every $g\in G$, $\rho(g)$ is the isometry induced by a linear map   $T_g$ preserving a sesquilinear form of signature $(1,\infty)$ defined on a Hilbert space. These linear maps are such that for every $g,l\in G$, 
$$T_gT_l=e^{i\alpha(gl,g,e)}T_{gl}. $$  
	
	\begin{prop}\label{levantamientolinealsiescofrontera}
		Let $G\xrightarrow{\rho}\iso(\mathbf{H}^m_\C)_o$ be a representation and  suppose $x\in \mathbf{H}^m_\C$ is a point with total orbit. If there exists $\omega$, an alternating $G$-invariant 1-cochain,  such that  $\partial\omega=\alpha$, where $\alpha$ is the 2-cocycle associated to $x$,  then 
		$\rho$ admits a lift to a representation $G\xrightarrow{\tilde{\rho}}U(1,m)$. 
	\end{prop}
	\begin{proof} 
		Let $T_g$ the map defined in the proof of \cref{gns}. Define $T'_g=e^{-i\omega(g,e)}T_g.$
		Observe that on one side $$\begin{array}{rcl}
			T'_gT'_h&=&e^{-i\omega(g,e)}e^{-i\omega(h,e)}T_gT_h\\
			&=&e^{-i\omega(g,e)}e^{-\omega(h,e)}e^{i\alpha(gh,g,e)}T_{gh}\\
			&=&e^{-i\omega(g,e)}e^{-i\omega(h,e)}e^{i\alpha(gh,g,e)}e^{i\omega(gh,e)}T'_{gh},
		\end{array}$$
		and that on the other side, 
		$$\alpha(gh,g,e)=\omega(g,e)-\omega(gh,e)+\omega(h,e).$$
		Therefore the map $g\mapsto T'_g$ is a homomorphism. 
	\end{proof}
	\begin{prop}\label{fijaunpuntoeslineal}
		Let $G\xrightarrow{\rho}\iso(\mathbf{H}^m_\C)_o$ be a representation and let $x\in\mathbf{H}^m_\C$. If $\rho$  fixes a point $y\in\partial \mathbf{H}^m_\C$, then there exists $\omega$, an alternating $G$-invariant $1$-cochain such that $\partial\omega=\alpha$, where $\alpha$ is  the 2-cocycle associated  to $x$. 
	\end{prop}
	\begin{proof} The continuity arguments will not be discussed, but it will be clear from the arguments used that they can be deduced. Define $\omega(g,l)=\Cart(y,\rho(g)x,\rho(l)x).$ Choose  $\tilde{x}$ a lift of $x$ and $\tilde{y}$ a lift of $y$ such that $B(\tilde{x},\tilde{y})>0$. Let $\tilde{\rho}$ be a linear lift of $\rho$ such that for every $g\in G$,    $\tilde{\rho}(g)(\tilde{y})=\theta_g\tilde{y}$, with $\theta_g>0$. Thus, 
		$$\begin{array}{rcl}
			\Cart(y,\rho(l)x,\rho(k)x)&=&	\\
			\Cart(\rho(l^{-1})y,x,\rho(l^{-1})\rho(k)x)&=&\\
			\Arg\Big(B(\theta_{l^{-1}}\tilde{y},\tilde{x})B\big(\tilde{x}, \tilde{\rho}({l^{-1}}{k})\tilde{x}\big)B\big(\tilde{\rho}({l^{-1}}{k})\tilde{x},\theta_{l^{-1}}\tilde{y}\big)\Big)&=&\\
			
			\Arg\Big(B(\tilde{x}, \tilde{\rho}(l^{-1}{k})\tilde{x})B(\tilde{x},\tilde{\rho}(k^{-1}{l})\tilde{y})\Big)&=&\\
			\Arg(B(\tilde{\rho}(l)\tilde{x}, \tilde{\rho}({k})\tilde{x}))	. 	
		\end{array}$$
		For every $x_1,x_2,x_3\in\mathbf{H}^m_\C$, $$|\Cart(x,y,z)|<\pi/2$$ and for $y\in\partial\mathbf{H}^m_\C,$ $$|\Cart(y,x_1,x_2)|\leq\pi/2.$$ Therefore, 
		
		$$ \begin{array}{rcl}
			
			\Arg(B(\tilde{\rho}(l)\tilde{x}, \tilde{\rho}({k})\tilde{x}))	-	\Arg(B(\tilde{\rho}(g)\tilde{x}, {\tilde{\rho}(k)}\tilde{x}))+
			\Arg(B(\tilde{\rho}(g)\tilde{x}, \tilde{\rho}({l})\tilde{x}))	
			&=&\\
			\Arg\Big({B\big(\tilde{\rho}(l)\tilde{x}, \tilde{\rho}({k})\tilde{x}\big){B\big( \tilde{\rho}({k})\tilde{x},\tilde{\rho}(g)\tilde{x}\big)}B\big(\tilde{\rho}(g)\tilde{x}, \tilde{\rho}({l})\tilde{x}\big)}\Big). 
		\end{array}$$
In other words, 
		$$\begin{array}{rcl}
			\partial \omega(g,l,k)&=&\\
			\Cart(y,\rho(l)x,\rho(k)x)-\Cart(y,\rho(g)x,\rho(k)x)+\Cart(y,\rho(g)x,\rho(l)x)&=&\\\alpha(g,l,k)	. 
		\end{array}$$

	\end{proof}
	The following corollary is a consequence of  propositions  \ref{levantamientolinealsiescofrontera} and    \ref{fijaunpuntoeslineal}.
	\begin{cor}\label{ellevantamientoescontinuo}
		Let $G\xrightarrow{\rho}\iso(\mathbf{H}^m_\C)_o$ be a representation and suppose  $x\in \mathbf{H}^m_\C $ has a total orbit. If $\rho$ fixes a point  at infinity, then $\rho$ admits an orbitally continuous  lift $\tilde{\rho}$ to $U(1,m)$. 
	\end{cor}

	\subsection{Non-elementary representations of $\iso(\mathbf{H}^1_\C)_o$}
	The technique  of studying non-elementary representations  through their restrictions to stabilizers of points at infinity can be tracked back to \cite{burger2005equivariant} and \cite{monod2014exotic}. In particular the results of this subsection follow the ideas of the latter. 

In this subsection the notation and general properties of non-elementary representations of $\text{Isom}(\mathbf{H}^1_\C)_o$ into a the group of holomorphic isometries of a infinite-dimensional complex hyperbolic space will be discussed. 

	Fix $\rho$ an irreducible   representation of $\iso(\mathbf{H}_\C^1)_o$ into the group of holomorphic isometries of $\hi_\C$. With an abuse of notation, it  will  be supposed often that $\rho$ is  defined on $SU(1,1)$. 
	By \cref{preservaeltipo},  $\rho(P)$ fixes a unique point $y_1\in\partial\hi_C$. The subgroup  $\{g(\lambda,0)\}_{\lambda>0}$ is abelian and again by \cref{preservaeltipo}, there is a unique $y_2\in\partial\hi_C$ such that $y_1$ and $y_2$ are the extremes of the axis shared by every element $\rho(g(\lambda,0))$. 
	
	Without lost of generality, fix a representative  $\eta_1$ (resp.  $\eta_2$) of $y_1$ (resp. $y_2$) such that $B(\eta_1,\eta_2)=1. $ 
	By  \cref{ellevantamientoescontinuo}, $\rho|_P$ admits a lift into an homomorphism $P\rightarrow U(B).$  With a small abuse of notation, the name  $\rho$ will be kept for this lift. Up to a multiplication by a continuous  homomorphism $P\xrightarrow{}\mathbf{S}^1$, it is possible to suppose that $B(\rho(\lambda,b)\eta_1,\eta_2)>0$, for every $\lambda>0$ and $b\in\R$. 
	
The following proposition is inspired by  Proposition 2.3 and the arguments in p.8 of  \cite{monod2014exotic}. 
	\begin{prop}\label{representacionenbloques}
		There exists a continuous isomorphism  $\chi:\R_{>0}\rightarrow\R_{>0}$ such that for every  $g=g(\lambda,b)\in P$, $\rho(g)=\rho(\lambda,b)$ can be represented by the transformation
		$$\begin{pmatrix}
			\chi(\lambda)&-\chi(\lambda)|c(\lambda,b)|^2/2+i\Delta(\lambda,b)&-\chi(\lambda)\langle\pi(\lambda,b)({\boldsymbol{\cdot}}),c(\lambda,b)\rangle\\
			0&\chi(\lambda)^{-1}&0\\
			0&c(\lambda,b)&\pi(\lambda,d)\end{pmatrix},$$ 
		with respect to the decomposition  $\C\eta_1\oplus\C\eta_2\oplus (\eta_1^\perp\cap\eta_2^\perp )$, where $\langle\boldsymbol{\cdot},\boldsymbol{\cdot}\rangle$ is the restriction of $B$ to $\eta_1^\perp\cap\eta_2^\perp$,  $|c(\lambda,b)|^2=\langle c(\lambda,b),c(\lambda,b)\rangle$ and where   for every $\lambda>0$ and $b,d\in\R$, 
		\begin{enumerate}
			\item 
			$\Delta(\lambda,b)\in\R$ and the map $(\lambda,b)\mapsto\Delta(\lambda,b)$ is continuous.  
			\item  $c(\lambda,b)\in\eta_1^\perp\cap \eta_2^\perp$ and the map $g(\lambda,b)\mapsto c(\lambda,b)$ is continuous. 
			\item $\pi(\lambda,b)$ is a unitary map of $\eta_1^\perp\cap \eta_2^\perp$ and the map $g(\lambda,b)\mapsto \pi(\lambda,b)$ is a strongly continuous unitary representation. 
			\item   
			$c(1,b+d)=c(1,b)+\pi(1,b)c(1,d). $
			\item $\Delta(\lambda,0)=0$ and $c(\lambda,0)=0$.
			\item $\chi(\lambda)\pi(\lambda,0)c(1,b)=c(1,\lambda^2 b)$. 
			\item $\chi(\lambda)^{2}\Delta(1,b)=\Delta(1,\lambda^2 b)$. 
			\item $-\Delta(1,b)=\Delta(1,-b)$.
			\item $
			Im(\langle c(1,d),c(1,b) \rangle)=\Delta(1,d-b)-\Delta(1,d)+\Delta(1,b)$.
			\item $Re(\langle c(1,d), c(1,b)\rangle)=-\frac{|c(1,-b+d)|^2}{2}+\frac{|c(1,b)|^2}{2}+\frac{|c(1,d)|^2}{2}. $

		\end{enumerate}
	\end{prop}
	\begin{proof}
		Let $c$ be the Busemann character associated to $\eta_1$.  By \cref{kerneldec} and the comments before it, for every $g\in P$, $\rho(g)\eta_1=\theta(g)\eta_1$, where $$c(\rho(g))=\ln(|\theta(g)|).$$
		
		Observe that $g(\lambda,b)=g(\lambda,0)	g(1,\lambda^{-1}b)$, therefore $$c(\rho(\lambda,b))=c(\rho(\lambda,0)).$$
		Thus,  the map $\lambda\mapsto e^{c(\rho(\lambda,0))}=\chi(\lambda)$ is a non-trivial  continuous  isomorphism $\R_{>0}\rightarrow\R_{>0}$ (see propositions  \ref{kerneldec} and  \ref{preservaeltipo}). 
		
		The points 1., 2., 3. and 4. are  consequences of: the  map $\rho$ is   a homomorphism to $U(B)$ which is orbitally continuous,  $\rho(P)$ preserves the line generated by $\eta_1$ and the comments before \cref{transtividad}. 
		
		For 5. observe that by construction $\Delta(\lambda,0)=0$ and $c(\lambda,0)=0$ because $\rho(\lambda,0)$ is hyperbolic and $\eta_1$ and $\eta_2$ are representatives of  the extremes of the axis preserved by it. 
		
		The points 6. and 7. are  consequences of the identity $$g(\lambda,0)g(1,b)g(\lambda^{-1},0)=g(1,\lambda^2b).$$
		
		For 8. observe that form  the identity  $g(1,b)g(1,-b)=g(1,0), $ $$\Delta(0)=\Delta(b)+\Delta(-b)-Im(\langle\pi(1,b)c(1,-b), c(1,b)\rangle)$$
		and by 4., $$Im(\langle\pi(1,b)c(1-b), c(b)\rangle)=-Im(\langle c(1,b),c(b)\rangle)=0.  $$ Therefore 
		$-\Delta(1,b)=\Delta(1,-b).$
		
		The points 9. and 10.  can be deduced from  4. and 8.  and the fact that $\rho$ is a homomorphism. Indeed, observe that $$\begin{array}{rcl}
			|c(1,-b+d)|^2&=&|c(1,b)|^2+|c(1,d)|^2+2 Re(\langle c(1,-b), \pi(1,-b)c(1,d)\rangle )\\
			&=&|c(1,b)|^2+|c(1,d)|^2-2 Re(\langle c(1,b), c(1,d)\rangle)
		\end{array}$$
		and that  $$\begin{array}{rcl}
			\Delta(1,d-b)    &=&\Delta(1,d)-\Delta(1,b)-Im(\langle \pi(1,-b)c(1,d),c(1,-b) \rangle) \\
			&=&\Delta(1,d)-\Delta(1,b)+Im(\langle c(1,d),c(1,b) \rangle).\end{array}$$
	\end{proof}
By 4., if $\pi$ and $c$ are restricted to $\{g(1,b)\}_{b\in\R}	$, then  $c$ is an affine cocycle associated to the representation $\pi$. In the rest of  the text, for an irreducible representation, the notation of the previous proposition is fixed. The conventions   $\Delta(1,b)=\Delta(b)$, $c(1,b)=c(b)$ and $\pi(1,b)=\pi(b)$ will be used from now on. 

	There will be often an abuse of notation:  $g(\lambda,b) $ (resp. $\rho(\lambda,b)$) will denote  either the isometry of $\mathbf{H}_\C^1$ (resp. $\hi_\C$) or   a representative in $SU(1,1,)$ (resp.  $U(1,\infty)). 
	$ It will be clear at all times what use of the notation is being made.  Also the symbols $\xi_i$ (resp. $\eta_i$) will be used either for the   points in $\partial\mathbf{H}_\C^1$ (resp.  $\partial\hi_\C$) or for their representatives. Again, this will not generate confusion.
	
	\begin{lem}\label{tesmenorque2}
		If $\chi(\lambda)=\lambda^t$, then $0<t\leq 2$.  
	\end{lem}
	\begin{proof}
		If $b\in\R\setminus\{0\}$, the transformation $g(1,b)$ is parabolic, therefore  $c(b)\neq0$ or $\Delta(b)\neq0$.  	The maps $b\mapsto c(b)$ and $b\mapsto \Delta(b)$ are  continuous and such that  $c(0)=0$ and $\Delta(0)=0$. For every $\lambda>0$, $\lambda^t|c(b)|=|c(\lambda^2b)|$ and $\lambda^{2t}\Delta(b)=\Delta(\lambda^2b)$, thus   $t>0$.
		
		Observe that $c(2b)=c(b)+\pi(b)c(b)$, therefore $2^{\frac{t}{2}}|c(b)|\leq 2|c(b)|$. Thus, if $c(b)\neq0$ for some (every) $b$, then $t\leq 2$. If this is not the case, by \cref{representacionenbloques}, the map $b\mapsto \Delta(b)$ is a (non-trivial) homomorphism. Therefore $$2^t\Delta(b)=\Delta(2b)=2\Delta(b)$$ and $t=1.$
	\end{proof}
	
	Define $$K(b)=-\frac{|c(b)|^2}{2}+i\Delta(b).$$
	Observe that 
	$$K(b)=\frac{B(\eta_2,\rho(1,b)\eta_1)}{|B(\eta_2,\rho(1,b)\eta_1)|}B(\rho(1,b)\eta_2,\eta_2),$$ therefore $\Arg(K(b))$  does not depend on  $\eta_1$ and $\eta_2$, the representatives of the extremes of the axis preserved by the maps $\rho(\lambda,0)$,  if the normalization condition $B(\eta_1,\eta_2)=1$ is imposed.  
	
		The following lemma  is an immediate consequence of \cref{representacionenbloques}. 
	\begin{lem}\label{propiedadesdeK}
		Given a non-elementary representation, 
		for every $\lambda>0$ and every $b\in\R$,      the following hold. 
		\begin{enumerate}{
			}{}
			\item $K(\lambda b)=\lambda^tK(b).$
			\item $K(-b)=\overline{K(b)}.$
			\item $K(b+d)=K(b)+K(d)+\langle c(d),c(-b)\rangle$.
		\end{enumerate} 
	\end{lem} 

	Define  the \textit{displacement} of $\rho$ as $t$ and denote it by $\ell(\rho)$.  
	The following lemma is similar to  Theorem B in \cite{monod2014exotic}. The proof there works for this particular case (see propositions  \ref{transtividad} and \ref{kerneldec}). 
	\begin{lem}
		For every  $g\in\iso(\mathbf{H}_\C^1)_o$, $\ell(\rho(g))=\ell(\rho)\ell(g)$. 
	\end{lem}
	
	In \cite{monod2014exotic}, among other things, the authors classified the irreducible representations $\iso(\mathbf{H}_\R^n)\xrightarrow{\rho}\iso(\mathbf{H}_\R^\infty)$. They showed that for every $0<t<1$ there exists a unique, up to a conjugation, irreducible representation $\rho_t$ such that for every $g\in\iso(\mathbf{H}_\R^n)$,   $\ell(\rho)=t$. For $t=1$ they showed that there is not an irreducible representation $\rho$ such that $\ell(\rho)=1$. 
	
	Every representation $\iso(\mathbf{H}_\R^2)\xrightarrow{\rho}\iso(\mathbf{H}_\R^\infty)$
	is in fact a linear representation into $O(1,\infty)$.  
	By Corollary 5.2 of \cite{monod2018notes}, $\rho_t$ restricted 
	to $\iso(\mathbf{H}_\R^2)_o $ remains non-elementary, thus it has an irreducible part. With a small abuse of notation, keep the notation $\rho_t$ for the irreducible representation. 
	There is a natural embedding  $O(1,\infty)<U(1,\infty)$ through complexification. In Proposition 5.10 of \cite{monod2018notes}, the author showed that the complexification of any irreducible representation of $\iso(\mathbf{H}_\R^n)$ into $O(1,\infty)$ remains irreducible. The proof there  works also for $\iso(\mathbf{H}_\R^2)_o$. 
	Thus the complexification of $\rho_t, 
$
	$$\iso(\mathbf{H}_\R^2)_o\xrightarrow{\rho_t^\C}U(1,\infty)$$ 
	is irreducible and such that 
	$\ell(\rho_t^\C)=t$. 
	
	Let $$\isocuno\xrightarrow{\overline{\Phi}}\iso(\mathbf{H}_\R^2)_o$$ be  the homomorphism of \cref{semultiplicapordos} and recall that if $g\in \iso(\mathbf{H}_\C^1)_o$, then $\ell(\overline{\Phi}(g))=2\ell(g)$. Therefore for every $t\in(0,1)$ and every $g\in \isocuno$, 
	$$\ell (\rho^\C_t\circ\overline{\Phi}(g))=t\ell(\overline{\Phi}(g))=2t\ell(g).$$
	This shows that for every $q\in(0,2)$ there exists an irreducible representation $$\isocuno\xrightarrow{\rho}\iso(\mathbf{H}^\infty_\C)_o$$ such that 
	$\ell(\rho)=q.$
	\begin{lem}\label{Deltaesceroparacomplexificaciones}
		If $\rho$ is     the complexification of an irreducible representation $$\isocuno\xrightarrow{} \iso(\hi_\R)_o, $$ then $\Delta(b)=0$, for every $b\in\R$. 
	\end{lem}
	\begin{proof}
		Observe that $\Arg(K(b))$ does not depend on the choice of representatives $\eta_1$, $\eta_2$  of the extremes of the axis preserved by the isometries $\rho(\lambda,0)$ as long as $B(\eta_1,\eta_2)=1$ (see the definition before \cref{propiedadesdeK}). Therefore if $\eta_1$ and $\eta_2$ are chosen in the totally real subspace that contains the representatives of the real hyperbolic subspace of $\hi_\C$ preserved by $\rho$, it is clear that $K(b)\in\R$. 	
	\end{proof}
	The following proposition follows some of the  ideas of \cite{monod2014exotic} and shows that there is not an irreducible representation $\isocuno\xrightarrow{\rho}\iso(\mathbf{H}_\C)_o$ such that $\ell(\rho)=2.$   
	\begin{prop}\label{tesdoseslineal}
		If $\rho$ is only supposed non-elementary and  such that $\ell(\rho)=2$, then $b\mapsto c(b)$ is a non-trivial linear map and $\Delta(b)=0$, for every $b\in\R$.  
	\end{prop}
	\begin{proof}
	The decomposition of an isometry $\rho(g)$ and the properties of the maps in  \cref{representacionenbloques} are still valid if $\rho$ is only supposed non-elementary. Observe that   	
		$$\begin{array}{rcl}
			2 |c(b)|&=&|c(2b)|\\
			&=&|c(b)+\pi(b)c(b)|\\
			&=&\leq |c(b)|+|\pi(b)c(b)|\\
			&\leq&2|c(b)|. \end{array}$$
		Therefore  $\pi(b)c(b)=c(b)$, for every $b\in\R$. Observe that for every $b,d\in\R$, 
		$$\begin{array}{rcl}
			\pi(b)c(d)+c(b)&=&c(b+d)\\
			&=&\pi(b+d)c(b+d)\\
			&=&\pi(b)\pi(d)\big(c(d)+\pi(d)c(b)\big)\\
			&=&\pi(b)c(d)+\pi(d)c(b).
		\end{array}$$
		Thus, for every $b,d\in\R$, $\pi(d)c(b)=c(b)$, or in other words, 
		$b\mapsto c(b)$ is a linear map. 
		This implies  that  for every $b,d\in\R$, 
		$Im(\langle c(b),c(d)\rangle )=0$, thus by \cref{representacionenbloques}, the map $b\mapsto \Delta(b)$ is linear, but $\Delta(2)=4\Delta(1)$. Therefore  $c(b)\neq 0$ and $\Delta(b)=0$, for every $b\neq 0$. 
	\end{proof}

	\begin{lem}\label{linealmenteindependiente}
		If $\ell(\rho)\neq1,2$, the  family $\{c(b)\}_{b\in\R\setminus 0}$ is $\C$-linearly independent. \end{lem}
	\begin{proof}
		Suppose $\sum^n a_ic(b_i)=0$ with $b_i\neq0$. Without lost of generality, suppose that $b_1>b_i$ for every $i\neq1$.  For every $d\in \R$, 
		$$\begin{array}{rcl}0&=&Re\big(\sum^n a_i \langle c(b_i), c(d)\rangle\big)\\
			&=&	\sum^n Re(a_i) Re(\langle c(b_i),c(d)\rangle) -Im(a_i) Im(\langle c(b_i),c(d)\rangle) \end{array}$$
		and 
		$$\begin{array}{rcl}
			0&=&Im\big(\sum^n a_i\langle c(b_i),c(d)\rangle\big) \\
			&=&\sum^n Re(a_i)Im(\langle c(b_i),c(d)\rangle) +Im(a_i)Re(\langle c(b_i),c(d)\rangle).
		\end{array}
		$$
		Consider an interval $(b_1,b_1+r)$ such that $0\not\in (b_1,b_1+r)$ and consider $d\in (b_1,b_1+r).$
		By \cref{representacionenbloques} there are   constants $C_0, C_1, D_0, D_1$ such that 
		$$C_0d^t+\sum^nRe(\overline{K(1)}a_i) (d-b_i)^t =C_1$$
		and 
		$$D_0d^t +\sum^n Im (\overline{K(1)}a_i) (d-b_i)^t
		=D_1.$$
		Thus, there exist  constants  $E_0 ,E_1$ such that for every $d\in (b_1,b_1+r)$, 
		$$E_0d^t+\sum^n \overline{K(1)}a_i(d-b_i)^t=E_1.$$
		
		After differentiating twice the previous equality with respect to $d$ in the interval  $(b_1,b_1+r)$, it follows that 
		$$t(t-1)E_0d^{t-2}+t(t-1)\sum^n \overline{K(1)}a_i((d-b_i)^{t-2} 
		=0.$$
		If $d\to b_1^+$, then $(d-b_1)^{t-2}$ is unbounded, but for every $i\neq1$, $(d-b_i)^{t-2}$ is bounded. Therefore $a_1=0$. 
		Repeating  the same argument, it is possible to show that  for every $i$, $a_i=0$.  
	\end{proof}
	Let $\sigma$ be  the isometry of $\mathbf{H}_\C^1$ represented, in the basis $\{\xi_1,\xi_2\}$,  by  
	$$s =\begin{pmatrix}
		0&i\\
		i&0\\
	\end{pmatrix}.$$ 
	The next proposition and the corollary after it follow the arguments  of Proposition 2.4 of  \cite{monod2014exotic}. 
	\begin{prop}\label{formulaparaA}
		The isometry $\rho(\sigma)$ can be represented  in the decomposition $\C\eta_1\oplus \C \eta_2\oplus(\eta_1^\perp\cap\eta_2^\perp)$,  by $$\begin{pmatrix}
			0&\nu^{-1}&0\\
			\nu&0&0\\
			0&0&A\end{pmatrix},$$ 
		for some $\nu>0$ and some unitary map  $A$ such that, for every $b\in\R$,  
		$$Ac(b)=\nu K(b)c(1,-1/b). $$
	\end{prop}
	\begin{proof}
		Observe that  $g(\lambda,0)s=sg(\lambda^{-1},0)$, therefore $\rho(\sigma)$ preserves the set $\{\eta_1,\eta_2\}$. If $\rho(\sigma)$ fixes it pointwise, then $\eta_1$ would be a $\isocuno$-fixed point (see \cref{transtividad}) which is a contradiction. Thus $\rho(\sigma)$ admits a linear representative 
		$$\begin{pmatrix}
			0&\nu^{-1}&0\\
			\nu&0&0\\
			0&0&A\end{pmatrix},$$ 
		with $\nu>0$ and $A$ a unitary map of $\eta_1^\perp\cap\eta_2^\perp.$
		
		Observe that as elements of $SU(1,1)$, 	
		$$s\cdot g({1,b})\cdot s\cdot g(1,{b}/{|b|^2})\cdot s=\begin{pmatrix}
			b^{-1}&-i\\0&b
		\end{pmatrix},$$
		thus, 
		$$\rho(\sigma)\rho(1,b)\rho(\sigma)\rho(1,{b}/{|b|^2})\rho(\sigma)=\rho(1/|b|, -b/b|).$$
		Notice that using the canonical linear representatives of the isometries in the  previous identity, on one side, 
		$$\begin{array}{rcl}
			\rho(\sigma)\rho(1,b)\rho(\sigma)\rho(1,{b}/{|b|^2})\rho(\sigma)(\eta_2)&=&
			\rho(\sigma)\rho(1,b)(\eta_2)\\&=&
			\rho(\sigma)\Big( K(b)\eta_1+\eta_2 +c(b)  \Big)\\&=&
			\nu^{-1}\eta_1+\nu K(b)\eta_2+Ac(b), 
		\end{array}
		$$
		and on the other side, by \cref{representacionenbloques} and \cref{propiedadesdeK}, 
		$$\begin{array}{rcl}
			\rho(1/|b|, -b/b|)(\eta_2)&=&\rho(1/|b|,0)\rho(1,-b)(\eta_2)\\
			&=&\rho(1/|b|,0)\Big(K(-b) \eta_1+\eta_2+c(-b)\Big)\\
			&=&|b|^{-t}K(-b)\eta_1+|b|^t\eta_2+\pi(1/|b|,0)c(-b)\\
			&=&K(-b/|b|)\eta_1+|b|^t\eta_2+|b|^tc(-1/b). 
		\end{array}$$
		Thus, there exist  a unitary complex number $\theta$ such that 
		$$\theta\Big(\nu^{-1}\eta_1+\nu K(b)\eta_2+Ac(b)\Big) =K(-b/|b|)\eta_1+|b|^t\eta_2+|b|^tc(-1/b). $$
		Observe that 
		$\theta\nu K(b)=|b|^t$, therefore, by \cref{propiedadesdeK}, $\nu K(b/|b|)=\theta^{-1}$. This implies that 
		$$\begin{array}{rcl}Ac(b)&=&\nu K(b/|b|)|b|^tc(-1/b)\\
			&=&\nu K(b)c(-1/b).\end{array}$$	
	\end{proof}
	\begin{cor}\label{larestriccionaPdetermina}
		The representation $\rho$ is determined by its restriction to $P$. 
	\end{cor}
	\begin{proof}
		The identity $g(\lambda,b)=g(1,\lambda b)g(\lambda,0)$ implies that  $c(\lambda,b)=\lambda^{-t}c(\lambda b).$ Let $W$ be the closed complex  vector space  generated by $\{c(b)\}_{b\in\R}$ in $\eta_1^\perp\cap\eta_2^\perp$. By  \cref{representacionenbloques} and \cref{formulaparaA}, 	$\C\eta_1\oplus\C\eta_2\oplus W$ is a closed and invariant complex subspace of signature $(1,\infty)$. Therefore,  as $\rho$ is irreducible, $W=\eta_1^\perp\cap\eta_2^\perp$. 
		
		Observe that, 
		$$\begin{array}{rcl}	
		|c(1)|^2&=&\langle Ac(1), A c(1)\rangle \\
		&=& \nu^2|K(1)|^2|c(1)|^2  .
		\end{array}$$ This and \cref{formulaparaA} show that $\rho(\sigma)$ is determined by $\rho|_P$. 	\end{proof}
	
	\section{A new family of representations}
	
		In this section a family of irreducible representations of $\isocuno$ is built,  that to the best of the author's knowledge were not known before. The way this  is done  is using a binary product for irreducible representations that is developed in  Subsection \ref{abc}. 
		
		 In  Subsection 2.1 it is defined a  complete invariant for irreducible 
		representations. 
		This invariant allows  to assure that this new representations are not equivalent to any of those  described in \cite{monod2018notes} and  \cite{monod2014exotic} (see the comments before \cref{Deltaesceroparacomplexificaciones} and before  \cref{cartandelapotencidelatautologica}). In particular all the irreducible representations $\rho$ such that $\ell(\rho)=1$ are classified.

	\subsection{Invariants for representations of $\isocuno$}
	In this subsection the theory of kernels of complex hyperbolic type  developed in \cite{monod2018notes} is used to find a complete invariant for irreducible representations of $\isocuno.$ 
	\begin{lem}\label{cartaneneldedimensionfinita}
		For every $y\in\mathbf{H}_\C^1$, 
		$$\lim\limits_{b\to\infty}\Cart\big(g(1,b)y,g(1,-b)y,y\big)=-\frac{\pi}{2}.$$ 
	\end{lem}
	\begin{proof}
		If $y$ is represented by $w=a\xi_1+\xi_2$, then $Re(a)>0$ and 
		$$\begin{array}{rcl}
			\Cart\big(g(1,b)y,g(1,-b)y,y\big)&=&\\
			\Arg\Big(B(g(1,2b)w,w)B(g(1,-b)w,w)^2\Big)&=&\\
			\Arg\Big( 2Re(a)(4Re(a)^2-b^2)+8Re(a)b^2+ i((4Re(a)^2-b^2)2b-8Re(a)^2b)  \Big).
		\end{array}$$
		Therefore, 
		$$\lim\limits_{b\to\infty}	\Cart\big(g(1,b)y,g(1,-b)y,y\big)=-\frac{\pi}{2}.$$
	\end{proof}
Suppose $\rho$ is defined on $SU(1,1)$. Let $y\in\mathbf{H}_\C^1$ and let 
$K$ be the maximal compact subgroup of $SU(1,1)$ that fixes $y$. Denote  $x\in\hi_\C$ the point fixed by $\rho(K)$ (see Proposition 5.8 and Remark 5.9 in \cite{monod2018notes}). Then there exists $0\leq s$ such that 
	for every $g_1,g_2\in SU(1,1)$, 
	$$s\Cart(g_1y,g_2y,y)=\Cart(\rho(g_1)x,\rho(g_2)x,x). $$
	This is a consequence of the fact that the action of $SU(1,1)$ on $\mathbf{H}_\C^1$ is doubly transitive (see Remark 2.5 in \cite{monod2018notes}). 
	
	Observe that $s\leq1$ because there exist $g_1,g_2\in SU(1,1)$ such that $$|\Cart(g_1y,g_2y,y)|$$ is  as close as desired to $\pi/2$ and $$|\Cart(\rho(g_1)x,\rho(g_2)x,x)|<\pi/2.$$ 
	
	\begin{lem}\label{cartandimensioninfinita}	
	If $x\in\hi_\C,$ then 
		$$\lim\limits_{b\to\infty}\Cart (\rho(b)x,\rho(-b)x,x)=\Arg(K(-1)).$$	
	Moreover,  if $K\leq SU(1,1)$,  the  stabilizer of $x$, is a maximal compact subgroup,  $y\in\mathbf{H}^1_\C$ is the point fixed by $K$ and  $0< s\leq 1$ is  such that 
		$$	s\Cart\big(g(1,b)y,g(1,-b)y,y\big)=\Cart (\rho(b)x,\rho(-b)x,x),$$ then 
		$\frac{s\pi}{2}=\Arg(K(1)).$
	\end{lem}
	\begin{proof}
		If $\tilde{x}=a\eta_1+\eta_2+u$ is a representative of $x$, then 
		$$\begin{array}{rcl}
			\Cart (\rho(1,b)x,\rho(1,-b)x,x)&=&\\
			\Arg\Big( B\big(\rho({1,2b})\tilde{x},\tilde{x}\big)
			B\big(\rho({1,-b})\tilde{x},\tilde{x}\big)^2\Big)&=&\\
			\Arg\Big(\Big(2Re(a)+K(2b)+\langle u,c(2b)\rangle +\langle c(2b)+\pi(2b)u,u\rangle\Big) \times&&\\
			\Big(2Re(a)+K(-b)+\langle u,c(b)\rangle +\langle c(-b)+\pi(-b)u,u\rangle\Big)^2\Big).
		\end{array}$$
		There exist   constants $C_1,C_2>0$ such that for every $b>0$,  
		$$\left|\langle u,c(b)\rangle +\langle c(-b)+\pi(-b)u,u\rangle\right|\leq C_1b^{\frac{t}{2}}+ C_2. $$ 
		Therefore, 
		$$\begin{array}{rcl}\lim\limits_{b\to\infty}\Cart (\rho(b)x,\rho(-b)x,x)&=&\lim\limits_{b\to\infty}\Arg\big( K(2b)K(-b)^2 \big)\\
			&=&\Arg(K(-1)).  
		\end{array}$$	
		The second claim is immediate from \cref{cartaneneldedimensionfinita}.
	\end{proof}
	Observe that the previous lemma, \cref{cartaneneldedimensionfinita} and the fact that the Cartan argument is left-invariant  imply   that neither $\Arg(K(1))$ nor $s$  depend on the choice of the point $x\in\hi_\C$ fixed by a maximal compact subgroup of $SU(1,1)$. The previous lemma shows also  that  $\Delta(1)\geq0$. 
	
	In view of the previous observations define $\Arg(\rho)$, the \textit{angular invariant of $\rho$}, as $\Arg(K(1))$. This invariant can be defined for non non-elementary representations not necessarily irreducible.  With this normalization, for a non-elementary representation $\rho$, $0\leq\Arg(\rho)\leq\frac{\pi}{2}.$ 
	
	\begin{prop}
If $\rho$ is non-elementary and $\Arg(\rho)=\pi/2$, then $\rho$ preserves a copy of $\mathbf{H}^1_\C.$
	\end{prop}
	\begin{proof}
		Observe that if $Re(K(1))=0$, then for every $b\in \R$, $c(b)=0.$
	\end{proof}
The previous proposition, which is trivial in this context,  is contained in the much more general Theorem 1.1 of  \cite{duchesne2021}. 

	In \cite{monod2018notes}, the author showed that if $G$ is a topological group and  $(\beta,\alpha)$ is a $G$-invariant  kernel of hyperbolic type, then for every $0<t<1$, $(\beta^t,t\alpha)$ is a $G$-invariant  kernel of hyperbolic type (see Theorem 1.12 of the aforementioned article). The author  also showed, in particular,  that if $(\beta,\alpha)$ is a kernel of hyperbolic type associated to the tautological action of $\iso(\mathbf{H}_\C^1)_o$ on $\mathbf{H}_\C^1$, then for every $0<t<1$, 
	$(\beta^t,t\alpha)$ induces (see \cref{gns}) a non-elementary representation 
	$\iso(\mathbf{H}_\C^1)_o\xrightarrow{\rho} \iso(\hi_\C)_o$   such that $\ell(\rho)=t $ (see  Theorem 1.15 and Lemma 2.2 of \cite{monod2018notes}). 
		The following proposition is  a direct consequence  of \cref{gns}	and \cref{cartandimensioninfinita}.
	\begin{prop}\label{cartandelapotencidelatautologica}
		Let $x\in\mathbf{H}_\C^1$ and let $0<t<1$. If $\rho$ is the irreducible part of the  non-elementary representation associated to the kernel $(\beta^t,t\alpha)$, where $(\beta,\alpha)$ is the kernel of hyperbolic type associated to $x$ and the tautological action of $\isocuno$ on $\mathbf{H}_\C^1$, then  $$\Arg(\rho)=\frac{t\pi}{2}. $$
	\end{prop}

	\begin{lem}\label{gelfandparaelgrupoparabolico}
If $x\in\hi_\C$ is represented by $\frac{1}{\sqrt{2}}(\eta_1+\eta_2)$, then the function of hyperbolic type associated to $x$ can be reconstructed from  $K(1)$ and $\ell(\rho)$. 
	\end{lem}
	\begin{proof}
		The transformation  $\rho(\sigma)$ is determined by the restriction of $\rho$ to the subgroup $P$ (see \cref{formulaparaA}). The claim  is that the restriction of $\rho$  to $P$ is entirely  determined by $K(1)$ and the parameter $t$. 
		
		For every
		$b,d\in\R$ and for every $\lambda,\gamma>0$, 
		$$\begin{array}{rcl}
			\Cart\big(\rho(\lambda,b)x,\rho(\gamma,d)x,x\big)&=&\\
			\Arg\Big( B\big(\rho(\lambda\gamma^{-1},\gamma^{-1} b-\lambda^{-1} d)(\eta_1+\eta_2),\eta_1+\eta_2\big) B\big(\rho(\gamma,d)(\eta_1+\eta_2),\eta_1+\eta_2\big)\times&&\\
			B\big(\rho(\lambda^{-1},-b)(\eta_1+\eta_2),\eta_1+\eta_2\big)\Big)&=&\\
			\Arg\Big( B\big(\rho(1,\lambda^{-1} b-\lambda^{-2}\gamma  d)(\eta_1+\eta_2),\rho(\lambda^{-1}\gamma,0)(\eta_1+\eta_2)\big)\times&&\\ B\big(\rho(1,\gamma^{-1}d)(\eta_1+\eta_2),\rho(\gamma^{-1},0)(\eta_1+\eta_2)\big)\times&&\\
			B\big(\rho(1,-\lambda b)(\eta_1+\eta_2),\rho(\lambda,0)(\eta_1+\eta_2)\big)\Big). 
		\end{array}$$
		Observe that, by \cref{propiedadesdeK},  the last term can be recovered knowing the values of  $K(1)$ and $t$. 
		For the same reasons,  $$\begin{array}{rcl}
			coshd(\rho(\lambda,b)x,x)&=&\\
			coshd(\rho(1,\lambda^{-1}b)x,\rho(\lambda^{-1},0)x)&=&\\
			\frac{1}{2}\left|B\big((1+K(\lambda^{-1}b)\eta_1+\eta_2,\lambda^{-t}\eta_1+\lambda^t\eta_2\big)\right|
		\end{array}$$
		can be also recovered from $K(1)$ and $t$. 
		Therefore the claim follows from \cref{gns} and the fact that the $P$-orbit of $\frac{1}{\sqrt{2}}(\eta_1+\eta_2)$ is total (see \cref{transtividad}). 
	\end{proof}
	\begin{teo}\label{clasificacionconK}
		Let $\rho_1$ and $\rho_2$ be two  irreducible representations of $\isocuno$ into  $\iso(\mathbf{H}_\C^\infty)_o$ such that $\ell(\rho_1)=\ell(\rho_2)$. 
		Then $\rho_1$ and $\rho_2$ are equivalent if, and only if, $\Arg(\rho_1)=\Arg(\rho_2)$.
	\end{teo}
	\begin{proof}
		Suppose  that $\rho_1(P)$ and $\rho_2(P)$ share the same fixed point in $\partial \hi_\C$ and that the families $\{\rho_1(\lambda,0)\}_{\lambda>0}$ and
	$\{\rho_2(\lambda,0)\}_{\lambda>0}$ preserve the same axis. 
	
		If $\rho_1$ and $\rho_2$ are equivalent, their restrictions to the group $P$ are equivalent. Therefore there  exists $T$ an isometry of $\hi_\C$ such that $T\rho_1|_P T^{-1}=\rho_2|_P$. 
		Notice that $T(\eta_i)=\eta_i$, where $\eta_i$  for $i=1,2$ are the extremes of the axis preserved by the isometries $\rho_i(\lambda,0)$. If    $$K_i(1)= \frac{-|c_i(1)|^2}{2}+i\Delta_i(1)= \frac{Q(\eta_2,\rho_i(1,1)\eta_1)}{|Q(\eta_2,\rho_i(1,1)\eta_1)|^2}B(\rho_i(1,1)\eta_2,\eta_2),$$ 
		then it is clear that $\Arg(K_1(1))=\Arg(K_2(1))$ because $\Arg(K_i(1))$ does not depend on the choice of the representatives of $\eta_i$, as long as the condition $B(\eta_1,\eta_2)=1$ is fulfilled. 
		
		Suppose $\Arg(K_1(1))=\Arg(K_2(1))$. After conjugating $\rho_1$ by an isometry $\rho_1(\gamma,0)$ if needed, it is possible to suppose that $K_1(1)=K_2(1)$. 
		 Let $x\in\mathbf{H}_\C^\infty$ be the point  represented by $$\frac{1}{\sqrt{2}}(\eta_1+\eta_2).$$ 
		Consider the respective kernels of hyperbolic type associated to $x$. By   \cref{gelfandparaelgrupoparabolico}, the representations $\rho_1|_P$ and $\rho_2|_P$ can be supposed identical, therefore by \cref{formulaparaA}, 
		$\rho_1$ and $\rho_2$ are equivalent. 
		
			Observe that the assuming that $\rho_1$ and $\rho_2$ have the same distinguished points at infinity is not restrictive.  Indeed, if $\rho$ is an irreducible  representation, by \cref{cartandimensioninfinita}, $\Arg(K(1))$ is invariant under conjugations and    given any $\eta_1, \eta_2\in\partial\hi_\C$, up to a conjugation,  it is possible to suppose that $\eta_1$ is the unique fixed  point of $\rho(P)$ and that  $\eta_2$ is the other extreme of the axis preserved by the family $\{\rho(\lambda,0)\}_{\lambda>0}.$ 
	\end{proof}

	\subsection{Extending certain representations of a parabolic subgroup}\label{abc}
	In this section a procedure to produce from two irreducible  representations with the same displacement a third one, which  in general will not be equivalent to any of the previous two,  is described. With this method a new family of non-equivalent  representations will be constructed and in particular, for irreducible representations with  displacement $1$ this family will be exhaustive. 
	
	The definition of this binary operation relies on the fact that with certain conditions, the representations of $P$ into $\iso(\mathbf{H}_\C^\infty)_o$ can be extended to representations of $SU(1,1)$. This is done using \cref{formulaparaA} as a definition for the image under the representation of the map $s$ (defined before the aforementioned proposition).  
	
	Fix  $\rho_1$ and $\rho_2$ two irreducible   representations of $\isocuno$ into  $\iso(\hi_\C)_o$  such that $\ell(\rho_1)=\ell(\rho_2)=t$. With the conventions of the previous section, suppose without lost of generality, that  $\rho_i$ share the two distinguished points $\eta_i\in\partial\hi_\C$. That is,  for every $g(\lambda,b)$, $\rho_i(\lambda,b)(\eta_1)=\eta_1$ and for every $g(\lambda,0)$, $\rho_i(\lambda,0)(\eta_2)=\eta_2.$
	
	In a matrix representation with respect to the decomposition $$\C\eta_1\oplus\C\eta_2\oplus(\eta_1^\perp\cap\eta_2^\perp),$$ by \cref{representacionenbloques},  $\rho_i(\lambda,b)$ has the shape, 
	$$\begin{pmatrix}
		\lambda^t&-\lambda^t|c_i(\lambda,b)|^2/2+i\Delta_i(\lambda,b)&-\lambda^t\langle\pi_i(\lambda,b)({\boldsymbol{\cdot}}),c_i(\lambda,b)\rangle\\
		0&\lambda^{-t}&0\\
		0&c_i(\lambda,b)&\pi_i(\lambda,d)\end{pmatrix},$$ 
	
	and the isometry $\rho_i(\sigma)$ has the representation 
	$$\begin{pmatrix}
		0&\nu_i^{-1}&0\\
		\nu_i&0&0\\
		0&0&A_i\end{pmatrix},$$ 
	where $\nu_i>0$ and, by \cref{formulaparaA},  $$A_ic_i(b)=\nu_i K_i(b)c_i(1,-1/b).$$

	Define a model for the hyperbolic space in the following way. Consider $\C^2$ as $\C\eta_1\oplus\C\eta_2$  and consider the Hilbert space $L=H_1\oplus H_2$,  where $H_i=\eta_1^\perp\cap\eta_2^\perp$.  Define  the form $Q$ in $\C\eta_1\oplus\C\eta_2\oplus L$ which is $\C$ linear in the first entry, antilinear in the second and  that is given by 
	\begin{enumerate}
		\item $Q|_{H_i}=\langle\boldsymbol{\cdot},\boldsymbol{\cdot}\rangle_i$.
		\item $Q(H_1,H_2)=0$. 
		\item $Q(\eta_i,H_j)=0$, for $i,j=1,2$. 
		\item $Q(\eta_i,\eta_i)=0,$ for $i=1,2$.
		\item $Q(\eta_1,\eta_2)=1$.
	\end{enumerate}
	This defines a  strongly non-degenerate form of signature $(1,\infty)$ in $$\C\eta_1\oplus\C\eta_2\oplus L. $$	
	
	Define  $c(b)=c_1(b)\oplus c_2(b)$ and $\pi(\lambda,b)=\pi_1(\lambda,b)\oplus\pi_2(\lambda,b)$, for every $b\in\R$ and $\lambda>0$. Observe that $\pi$ is a unitary representation of the group $P$ on  $L$.  
	
	Define $\rho(\lambda,0)$ as the isometry  represented by 
	$$\begin{pmatrix}
		\lambda^t&0&0\\
		0&\lambda^{-t}&0\\
		0&0&\pi(\lambda,0)\end{pmatrix}$$ 
	and  $\rho(1,b)$ as the isometry represented by 
	$$\begin{pmatrix}
		1&-|c(b)|^2/2+i\Delta(b)&-\langle\pi(b)({\boldsymbol{\cdot}}),c(b)\rangle\\
		0&1&0\\
		0&c(b)&\pi(b)\end{pmatrix},$$ 
	where $\Delta(b)=\Delta_1(b)+\Delta_2(b)$. Observe that the  transformations $\rho(\lambda,0)$ and $\rho(1,b)$ are isometries of $\hi_\C,$
the hyperbolic space associated to $\C\eta_1\oplus\C\eta_2\oplus L $ and $Q$. 	
	Denote   $K(b)=-|c(b)|^2/2+i\Delta(b)$ and 
	$K_i(b)=-|c_i(b)|^2/2+i\Delta_i(b)$.   
	
		The next  proposition is a consequence of  \cref{representacionenbloques} and \cref{propiedadesdeK}.
	\begin{prop}\label{propiedadesdeKyDelta}
		If $c$, $\pi$, $K$ and $K_i$ are defined as above, then for every $\lambda>0$ and $b,d\in\R$, the following properties hold. 
		\begin{enumerate}{
			}{}
			\item $
			Im(\langle c(b),c(d) \rangle)=\Delta(b-d)-\Delta(b)+\Delta(d)$.
			\item $Re(\langle c(b), c(d)\rangle)=-\frac{|c(b-d)|^2}{2}+\frac{|c(b)|^2}{2}+\frac{|c(d)|^2}{2}. $
			\item 	$K(b)=K_1(b)+K_2(b)$.
			\item $K(\lambda b)=\lambda^tK(b).$
			\item $K(-b)=\overline{K(b)}.$
			\item $K(b+d)=K(b)+K(d)+\langle c(d),c(-b)\rangle$.
			\item $\pi(\lambda,0)c(b)=\lambda^{-t}c(\lambda^2b).$
			\item $c(b+d)=c(b)+\pi(b)c(d).$
		\end{enumerate} 	
	\end{prop}
		
	\begin{lem}\label{kessiempredistintodecero}
		For every $b\neq0$, $K(b)\neq 0$. 		
	\end{lem}
	\begin{proof}
		Suppose $K(1)=0$. 	By \cref{cartandimensioninfinita}, $\Delta_i(1)\geq0$, therefore $\Delta_i(1)=0$. The isometries $\rho_i(1,1)$ are parabolic, thus  $c_i(1)\neq0$, which is a contradiction. 
	\end{proof}
	
	Observe that $g(\lambda,0)g(1,b)=g(\gamma,0)g(1,d)$ if, and only if, $\lambda=\gamma$ and $b=d,$ and that $g(\lambda,0)g(1,b)=g(1,\lambda^2b)g(\lambda,0)$. 
	It will be shown that with the formulas for  $\rho(\lambda,0)$ and $\rho(1,b)$ it  is possible to define an homomorphism on $P$. 
	\begin{lem}\label{computohorrible1}
		For every $\lambda,\gamma>0$ and $b,d\in\R$, the following identities hold. 
		\begin{enumerate}
			\item
			$	K(\lambda^{-1}\gamma^{-1}b+d)=	K(d)+\gamma^{-t}K(\lambda^{-1}b)+\langle \pi(\gamma,0)c(\gamma^{-1}d),c(-\lambda^{-1}b)\rangle.$
			
			\item 		$	\pi(\lambda^{-1}b)\pi(\gamma,0)c(\gamma^{-1}d)+\gamma^{-t}c(\lambda^{-1}b) =
			\pi(\gamma,0)c(\lambda^{-1}\gamma^{-2}b+\gamma^{-1}d).	$
			
		\end{enumerate}	
	\end{lem}
	\begin{proof}
		By \cref{propiedadesdeKyDelta}, 
		$$\begin{array}{rcl}
			K(\lambda^{-1}\gamma^{-1}b+d)&=&\\
			K(d)+K(\lambda^{-1}\gamma^{-1}b)+\langle c(d),c(-\lambda^{-1}\gamma^{-1}b)\rangle
			&=&\\
			K(d)+K(\gamma^{-1}\lambda^{-1}b)+\langle\gamma^{-t/2} c(d),\gamma^{t/2}c(-\gamma^{-1}\lambda^{-1}b)\rangle
			&=&\\
			K(d)+K(\gamma^{-1}\lambda^{-1}b)+\langle \pi(\gamma^{1/2},0)c(\gamma^{-1}d),\pi(\gamma^{-1/2},0)c(-\lambda^{-1}b)\rangle	&=&\\
			K(d)+\gamma^{-t}K(\lambda^{-1}b)+\langle \pi(\gamma,0)c(\gamma^{-1}d),c(-\lambda^{-1}b)\rangle
		\end{array}$$
		and  $$\begin{array}{rcl}
			\pi(\lambda^{-1}b)\pi(\gamma,0)c(\gamma^{-1}d)+\gamma^{-t}c(\lambda^{-1}b) &=&\\
			\gamma^{-t}\pi(\lambda^{-1}b)c(\gamma d)+\gamma^{-t}c(\lambda^{-1}b)&=&\\
			\gamma^{-t}c(\lambda^{-1}b+\gamma d)&=&\\
			\pi(\gamma,0)c(\lambda^{-1}\gamma^{-2}b+\gamma^{-1}d).	
		\end{array}$$
	\end{proof}
	\begin{lem}\label{computohorrible2}
		For every $\gamma,\lambda>0$, $b,d\in \R$ and $u\in\eta_1^\perp\cap\eta_2^\perp, $
		$$\begin{array}{rcl}	\gamma^t\langle u,c(-\gamma^{-1}d)\rangle+\langle u,\pi(-\gamma^{-1}d)\pi(\gamma^{-1},0)c(-\lambda^{-1}b)\rangle&=&\\\gamma^t\langle u,c(1,-\lambda^{-1}\gamma^{-2}b-\gamma^{-1}d)\rangle.\end{array}$$
	\end{lem}
	\begin{proof}
		By \cref{propiedadesdeKyDelta}, 
		$$\begin{array}{rcl}
			\gamma^t\langle u,c(-\gamma^{-1}d)\rangle+\langle u,\pi(-\gamma^{-1}d)\pi(\gamma^{-1},0)c(-\lambda^{-1}b)\rangle&=&\\
			\gamma^t\langle u,c(-\gamma^{-1}d)\rangle+\gamma^t\langle u,\pi(-\gamma^{-1}d)c(-\lambda^{-1}\gamma^{-2}b)\rangle &=&\\
			\gamma^t\langle u,c(1,-\lambda^{-1}\gamma^{-2}b-\gamma^{-1}d)\rangle. 
		\end{array}$$
		
	\end{proof}

	\begin{prop}\label{bylambdasonhomomorfismos}
		The map $g(\lambda,b)\mapsto \rho(\lambda,0)\rho(1,\lambda^{-1}b)$ is a homomorphism and for every $x\in\hi_\C$, the map 
		$g(\lambda,b)\mapsto \rho(\lambda,0)\rho(1,\lambda^{-1}b)x$ is continuous. 
	\end{prop}
	\begin{proof}
		Observe that $$\begin{array}{rcl}
			g(\lambda,b)g(\gamma,d)&=&g(\lambda,0)g(1,\lambda^{-1}b)g(\gamma,0)g(1,\gamma^{-1}d)\\&=&g(\lambda,0)g(\gamma,0)g(1,\lambda^{-1}\gamma^{-2}b)g(1,\gamma^{-1}d)\\
			&=&g(\lambda\gamma,0)g(1,\lambda^{-1}\gamma^{-2}b+\gamma^{-1}d)	.
		\end{array}
		$$
		Therefore the  first  claim of the proposition  is that $$\rho(\lambda,0)\rho(1,\lambda^{-1}b)\rho(\gamma,0)\rho(1,\gamma^{-1}d)=\rho(\lambda\gamma,0)\rho(1,\lambda^{-1}\gamma^{-2}b+\gamma^{-1}d)	.$$
		It is clear that $\lambda\mapsto \rho(\lambda,0)$ is a homomorphism, thus to show the previous claim is equivalent to show that 
	 $$\rho(1,\lambda^{-1}b)\rho(\gamma,0)\rho(1,\gamma^{-1}d)=\rho(\gamma,0)\rho(1,\lambda^{-1}\gamma^{-2}b+\gamma^{-1}d)	.$$
		This will be done by comparing the columns of the matrix representations of both sides of the identities with respect to the decomposition $\C\eta_1\oplus\C\eta_2\oplus L$.
		
		It is clear from the definition of $\rho$ that for $\eta_1$, 
		$$\begin{array}{rcl}\rho(1,\lambda^{-1}b)\rho(\gamma,0)\rho(1,\gamma^{-1}d)\eta_1&=&\gamma^t\eta_1\\&=&\rho(\gamma,0)\rho(1,\gamma^{-2}b+\gamma^{-1}d)\eta_1	.\end{array}$$
		
		By \cref{propiedadesdeKyDelta} and \cref{computohorrible1}, for $\eta_2$, 
		$$\begin{array}{rcl}\rho(1,\lambda^{-1}b)\rho(\gamma,0)\rho(1,\gamma^{-1}d)\eta_2&=&\\
			\rho(1,\lambda^{-1}b)\rho(\gamma,0)\Big(K(\gamma^{-1}d)\eta_1+\eta_2+c(\gamma^{-1}d)\Big)&=&\\
			\rho(1,\lambda^{-1}b)\Big(\gamma^tK(\gamma^{-1}d)\eta_1+\gamma^{-t}\eta_2+\pi(\gamma,0)c(\gamma^{-1}d)\Big)&=&\\
			\Big(K(d)+\gamma^{-t}K(\lambda^{-1}b)+\langle \pi(\gamma,0)c(\gamma^{-1}d),c(-\lambda^{-1}b)\rangle \Big)	\eta_1+\\ \gamma^{-t}\eta_2+\pi(\lambda^{-1}b)\pi(\gamma,0)c(\gamma^{-1}d)+\gamma^{-t}c(\lambda^{-1}b)&=&\\

			K(\lambda^{-1}\gamma^{-1}b+d)\eta_1+\gamma^{-t}\eta_2+\pi(\gamma,0)c(\lambda^{-1}\gamma^{-2}b+\gamma^{-1}d) &=&\\
			
			\rho(\gamma,0)\Big(K(\lambda^{-1}\gamma^{-2}b+\gamma^{-1}d)\eta_1+\eta_2+c(\lambda^{-1}\gamma^{-2}b+\gamma^{-1}d) \Big)&=&\\
			\rho(\gamma,0)\rho(1,\lambda^{-1}\gamma^{-2}b+\gamma^{-1}d)\eta_2	.\end{array}$$
		And last, by \cref{propiedadesdeKyDelta} and \cref{computohorrible2},  for $u\in\eta_1^\perp\cap\eta_2^\perp$, 
		$$\begin{array}{rcl}\rho(1,\lambda^{-1}b)\rho(\gamma,0)\rho(1,\gamma^{-1}d)u&=&\\
			\rho(1,\lambda^{-1}b)\rho(\gamma,0)\big(\langle u,c(-\gamma^{-1}d)\rangle\eta_1 +\pi(\gamma^{1}d)u\big)&=&\\
			\rho(1,\lambda^{-1}b)\big(\gamma^t\langle u,c(-\gamma^{-1}d)\rangle\eta_1 +\pi(\gamma,0)\pi(\gamma^{-1}d)u\big)&=&\\
			
			\Big(\gamma^t\langle u,c(-\gamma^{-1}d)\rangle+\langle\pi(\gamma,0)\pi(\gamma^{-1}d)u,c(-\lambda^{-1}b)\rangle \Big)\eta_1 +\\
			\pi(\lambda^{-1}b)\pi(\gamma,0)\pi(\gamma^{-1}d)u&=&\\
			
			\Big(\gamma^t\langle u,c(-\gamma^{-1}d)\rangle+\langle u,\pi(-\gamma^{-1}d)\pi(\gamma^{-1},0)c(-\lambda^{-1}b)\rangle \Big)\eta_1 +\\
			\pi(\gamma, \gamma^{-1}\lambda^{-1}b+d)u&=&\\
			
			\gamma^t\langle u,c(1,-\lambda^{-1}\gamma^{-2}b-\gamma^{-1}d)\rangle\eta_1+\pi(\gamma,0)\pi(\lambda^{-1}\gamma^{-2}b+\gamma^{-1}d)u &=&\\
			
			\rho(\gamma,0)\Big(\langle u,c(1,-\lambda^{-1}\gamma^{-2}b-\gamma^{-1}d)\rangle\eta_1+\pi(\lambda^{-1}\gamma^{-2}b+\gamma^{-1}d)u\Big) &=&\\
			\rho(\gamma,0)\rho(1,\lambda^{-1}\gamma^{-2}b+\gamma^{-1}d)u	.\end{array}$$
		Therefore the map $$g(\lambda,b)\mapsto \rho(\lambda,0)\rho(1,\lambda^{-1}b)\in \iso(\mathbf{H}_\C^\infty)_o$$ is a homomorphism. 
		
		For the second claim of the proposition  it is enough to show that for every $x\in\hi_\C$, the map $g(\lambda,b)\mapsto \rho(\lambda,0)\rho(1,\lambda^{-1}b)x$ is continuous around the identity in $P$. 
		Suppose $g_i=g(\lambda_i,b_i)\to Id$ in $P$, then 
		$$\frac{1}{4}|B\big(g(\lambda_i,b_i)(\eta_1+\eta_2),\eta_1+\eta_2\big)|^2=\frac{1}{4}\big((\lambda_i+\lambda_i^{-1})^2+b_i^{2}\big)\to 1,$$
		or equivalently,   $$(\lambda_i-\lambda_i^{-1})^2+b_i^2\to 0.$$  
		Therefore $\lambda_i\to1$ and $b_i\to 0$. 
		If $x=\alpha\eta_1+\beta\eta_2 + u$ is such that $Q(x,x)=1$, then 
		$$\begin{array}{rcl}\rho(\lambda_i,0)\rho(1,\lambda_i^{-1}b_i)x&=&\\
			\rho(\lambda_i,0)\big(\alpha+ \beta K(\lambda_i^{-1}b_i)+\langle u,c(-\lambda_i^{-1}b_i)\rangle\big)\eta_1+&&\\
			\rho(\lambda_i,0)\Big(\beta \eta_2+\beta c(\lambda_i^{-1}b_i)+\pi(\lambda_i^{-1}b_i)u  \Big)&=&\\
			\lambda^t_i\big(\alpha+ \beta K(\lambda_i^{-1}b_i)+\langle u,c(-\lambda_i^{-1}b_i)\rangle\big)\eta_1+	\lambda_i^{-t}\beta \eta_2+&&\\
			\pi(\lambda_i,0)\big(\beta c(\lambda_i^{-1}b_i)+\pi(\lambda_i^{-1}b_i)u \big)&=&\\
			\big(\lambda^t_i\alpha+ \beta K(b_i)+\langle u,\pi(\lambda_i^{-1},0)c(-\lambda_i b_i)\rangle\big)\eta_1+	\lambda_i^{-t}\beta\eta_2+&&\\
			\lambda_i^{-t}\beta c(\lambda_ib_i)+\pi(\lambda_i,0)\pi(\lambda_i^{-1}b_i)u.
		\end{array}$$
		Therefore, since  $$g(\lambda,b)\mapsto\pi(\lambda,b)=\pi_1(\lambda,b)\oplus\pi_2(\lambda,b)$$ is orbitally continuous,  
		$$\begin{array}{rcl}
			\lim\limits_{i\to\infty}\left|B\big(  \rho(\lambda_i,0)\rho(1,\lambda_i^{-1}b_i)x,x \big)\right|&=&\\
			\lim\limits_{i\to\infty}|\overline{\beta}	\big(\lambda^t_i\alpha+ \beta  K(b_i)+\langle u,\pi(\lambda^{-1},0)c(-\lambda_i b_i)\rangle\big)+\overline{\alpha}\lambda_i^{-t}\beta+&&\\
			\langle 	\lambda_i^{-t}\beta c(\lambda_ib_i)+\pi(\lambda_i,0)\pi(\lambda_i^{-1}b_i)u, u\rangle|&=&\\
			|\overline{\beta}\alpha +\overline{\alpha}\beta +\langle u,u\rangle|	&=&1.
		\end{array}$$
	\end{proof}
	Now it is possible to define the representation  $P\xrightarrow{\rho}\iso(\mathbf{H}^\infty_\C)_o$ given by  $$\rho(\lambda,b)=\rho(\lambda,0)\rho(1,\lambda^{-1}b).$$ The next  results are devoted to prove that $\rho$ can be extended to a homomorphism defined on $SU(1,1)$. 
	\begin{lem}\label{unicopuntofijo}
		The only point fixed  in $\hi_\C\cup\partial\hi_\C$ by $\rho$   is $\eta_1$. 
	\end{lem}
	\begin{proof}
		The isometries $\rho(\lambda,0)$ are hyperbolic by construction, therefore the only other candidate to be fixed by $\rho$ is $\eta_2$, but $\rho(1,b)$ does not fix it  because $K(1)\neq0$ (see \cref{kessiempredistintodecero}).  
	\end{proof}
	Observe that 
	$$K(b)=\frac{Q(\eta_2,\rho(1,b)\eta_1)}{|Q(\eta_2,\rho(1,b)\eta_1)|^2}Q(\rho(1,b)\eta_2,\eta_2)$$  is also true for the representation $\rho$. After a conjugation by an isometry $\rho(\gamma,0)$ if necessary, 
	assume  that $|K(1)|=1$. Notice  that this conjugation does not change  the argument of $K(1)$. 
	
	The following is the uniqueness  part of the GNS construction (see Theorem C.1.4 of  \cite{propertyTbekka}). 
	\begin{lem}\label{defineunmapeounitario}
		Let $X$ be a set and let $H$ be a Hilbert space. Suppose $f$ and $g$ are two functions $X\rightarrow H$  such that their images are total in $H$. If  for  every $x,y\in X$, $\langle f(x),f(y)\rangle =\langle g(x),g(y)\rangle$,
		then there exists $A$,  a unitary map of $H$, such that  $Af(x)= g(x)$.
	\end{lem}

	\begin{prop}\label{definiciondeA} The map   	$$Ac(b)=K(b)c(-1/b)$$ defines a unitary map in $L'=\overline{\langle \{c(b)\}_{b\in\R}\rangle}$ such that $A^2=Id.$
	\end{prop}
	\begin{proof}
		Due to \cref{defineunmapeounitario}, it is enough to show that $\langle Ac(b),Ac(d)\rangle=\langle c(b), c(d)\rangle$. Suppose $b\neq d$. By \cref{propiedadesdeKyDelta}, on one side, 
		$$\begin{array}{rcl}
			\langle A(c(b)),A(c(d))\rangle
			&=& \\
			K(b)\overline{K(d)}\langle c(-{1}/{b}),c(-{1}/{d})\rangle
			&=&\\	K(b)\overline{K(d)}\Big(-\frac{|c(-1/b+1/d)|^2}{2}+\frac{|c(1/b)|^2}{2}+\frac{|c(1/d)|^2}{2}\Big)+&&\\ K(b)\overline{K(d)}i\Big(\Delta(-1/b+1/d)-\Delta(-1/b)+\Delta(-1/d) \Big)&=&\\
			|b|^t|d|^tK\left(\frac{b}{|b|}\right)K\left(-\frac{d}{|d|}\right)\Big(-\frac{|d-b|^t}{|bd|^t}+\frac{1}{|b|^t}+\frac{1}{|d|^t}\Big)\frac{|c(1)|^2}{2}+&&\\
			|b|^t|d|^tK\left(\frac{b}{|b|}\right)K\left(-\frac{d}{|d|}\right)i\Big(\frac{|b-d|^t(b-d)|bd|}{|bd|^tbd|b-d|} + \frac{b}{|b|^t|b|}-\frac{d}{|d|^t|d|}\Big)\Delta(1)
			&=&\\
			
			K\left(\frac{b}{|b|}\right)K\left(-\frac{d}{|d|}\right)\Big(-{|d-b|^t}+{|d|^t}+{|b|^t}\Big)\frac{|c(1)|^2}{2}+&&\\
			
			K\left(\frac{b}{|b|}\right)K\left(-\frac{d}{|d|}\right)i\Big(\frac{|b-d|^t(b-d)|bd|}{bd|b-d|} + \frac{|d|^tb}{|b|}-\frac{|b|^td}{|d|}\Big)\Delta(1).
		\end{array}$$
		On the other side, 
		$$\begin{array}{rcl}
			\langle c(b),c(d)\rangle&=&\\
			-\frac{|c(b-d)|^2}{2}+\frac{|c(b)|^2}{2}+\frac{|c(d)|^2}{2}
			+i\big(\Delta(b-d)-\Delta(b)+\Delta(d)\big)&=&\\
			\big(-|b-d|^t+|b|^t+|d|^t\big)\frac{|c(1)|^2}{2}+i\big(\frac{|b-d|^t(b-d)}{|b-d|}-\frac{|b|^tb}{|b|}+\frac{|d|^td}{|d|}\big)\Delta(1)	.
		\end{array}$$
		There are three  cases to analyse:
		1) $b>0>d$, 2) $b> d>0$ and 3) $b< d<0$. 
		
		1)	If $b>0>d$, 	
		$$\begin{array}{rcl}
			K\left(\frac{b}{|b|}\right)K\left(-\frac{d}{|d|}\right)&=&K(1)^2
		\end{array}$$
		and 
		$$\begin{array}{rcl}
			\frac{|b-d|^t(b-d)|bd|}{bd|b-d|} + \frac{|d|^tb}{|b|}-\frac{|b|^td}{|d|}&=&
			-|b-d|^t+|b|^t+|d|^t	. 	\end{array}$$
		
		Therefore 
		$$\begin{array}{rcl}
			\langle Ac(b), Ac(d)\rangle &=&	\\
			K(1)^2	\big(	-|b-d|^t+|b|^t+|d|^t\big)\left(\frac{|c(1)|^2}{2}+i\Delta(1)\right)&=&\\
			\big(-|b-d|^t+|b|^t+|d|^t\big)K(1)^2(-\overline{K(1)})&=&\\
			-	\big(-|b-d|^t+|b|^t+|d|^t\big)K(1)
		\end{array}$$
		and 
		$$\begin{array}{rcl}
			\langle c(b),c(d)\rangle &=&\\
			\big(	-|b-d|^t+|b|^t+|d|^t\big)\big(\frac{|c(1)|^2}{2}-i\Delta(1)\big)&=&\\
			-\big(-|b-d|^t+|b|^t+|d|^t\big)K(1).
		\end{array}$$
		
		2)	If $b>d>0$, 
		$$\begin{array}{rcl}
			K\left(\frac{b}{|b|}\right)K\left(-\frac{d}{|d|}\right)&=&1
		\end{array}$$
		and 
		$$\begin{array}{rcl}
			\frac{|b-d|^t(b-d)|bd|}{bd|b-d|} -\frac{|b|^td}{|d|}+ \frac{|d|^tb}{|b|}&=&
			|b-d|^t-|b|^t+|d|^t. 	\end{array}$$
		Thus 
		$$\begin{array}{rcl}
			\langle A(c(b)),A(c(d))\rangle
			&=& \\
			\big(-{|d-b|^t}+{|d|^t}+{|b|^t}\big)\frac{|c(1)|^2}{2}+
			\big(	|b-d|^t+|d|^t-|b|^t\big)\Delta(1)
		\end{array}$$
		and 
		$$\begin{array}{rcl}
			\langle c(b),c(d)\rangle
			&=& \\
			\big(-{|d-b|^t}+{|d|^t}+{|b|^t}\big)\frac{|c(1)|^2}{2}+
			i\big(|b-d|^t-|b|^t+|d|^t\big)\Delta(1).
		\end{array}$$
		3) If $b< d<0$, 
		$$\begin{array}{rcl}
			K\left(\frac{b}{|b|}\right)K\left(-\frac{d}{|d|}\right)&=&1
		\end{array}$$
		and 
		$$\begin{array}{rcl}
			\frac{|b-d|^t(b-d)|bd|}{bd|b-d|} + \frac{|d|^tb}{|b|}-\frac{|b|^td}{|d|}&=&
			-|b-d|^t-|d|^t+|b|^t	. 	\end{array}$$
		Therefore 
		$$\begin{array}{rcl}
			\langle A(c(b)),A(c(d))\rangle &=&\\\big(-{|d-b|^t}+{|d|^t}+{|b|^t}\big)\frac{|c(1)|^2}{2}+i\big(-|b-d|^t-|d|^t+|b|^t\big)\Delta(1)
		\end{array}	$$
		and 
		$$	\begin{array}{rcl}
			\langle c(b),c(d)
			\rangle &=&\\
			\big(-|b-d|^t+|b|^t+|d|^t\big)\frac{|c(1)|^2}{2}+i\big(-|b-d|^t+|b|^t-|d|^t\big)\Delta(1).
		\end{array}$$
The case $b=d$ is an immediate consequence of \cref{propiedadesdeKyDelta}.

		By \cref{defineunmapeounitario}, $A$ induces a unitary map  on $L$. 
		And last, observe  that $$\begin{array}{rcl}
			A^2(c(b))&=&		K(b)K(-1/b)c(b)\big)\\&=&
			|b|^t	K(\frac{b}{|b|})\frac{1}{|b|^t}K(-\frac{b}{|b|})c(b)\\
			&=&c(b).
		\end{array}$$
		
	\end{proof}	
	Consider now $\hi_\C$ as the hyperbolic space associated to $\C\eta_1\oplus\C\eta_2\oplus L'$ and the restriction of  the form $Q$ defined in  the beginning of this section. 
	Denote by  $\tilde{\sigma}\in \iso(\hi_\C)_o$ the order two  transformation  represented by 
	$$\begin{pmatrix}
		0&1&0\\
		1&0&0\\
		0&0&A\\
	\end{pmatrix}.$$
	The claim is that  the representation $\rho$ can be extended to a representation of $SU(1,1)$ using $\tilde{\sigma}$. That is to say,  if   $g(1,b)$, with $b\in\R$, $g(\lambda,0)$, with $\lambda>0$, and $$s=\begin{pmatrix}
		0&i\\
		i&0
	\end{pmatrix}$$ are understood  as elements of $SU(1,1)$, then the map defined by : \begin{enumerate}
		\item $T(g(\lambda,0))=\rho(\lambda,0)$, 
		\item $T(-g(\lambda,0))=\rho(\lambda,0)$,
		\item $T(g(1,b))=\rho(1,b), $
		\item $T(s)=\tilde{\sigma}$,
	\end{enumerate}
	where $\rho(\lambda,0)$, $\rho(1,b)$ and $\tilde{\sigma}$ are interpreted as elements of $\iso(\hi_\C)_o$, is a homomorphism. 
	
	In order to prove that $T$ is a homomorphism it is enough to show that $T$ is coherent with the relations of \cref{relacionesygeneradores}, that is to show that, for $\lambda>0$ and $b\in\R$, 
	\begin{enumerate}
		\item $\rho(\lambda,0)=\tilde{\sigma}\rho(1,\lambda^{-1})\tilde{\sigma}\rho(1,\lambda)\tilde{\sigma}\rho(1,\lambda^{-1}).$
		\item $\rho(\lambda,0)=\tilde{\sigma}\rho(1,-\lambda^{-1})\tilde{\sigma}\rho(1,-\lambda)\tilde{\sigma}\rho(1,-\lambda^{-1}).$	
		\item $\lambda\mapsto \rho(\lambda,0)$ is a homomorphism. 
		\item $b\mapsto\rho(1,b) $ is a homomorphism. 
		\item $\rho(\lambda,0)\rho(1,b)\rho(\lambda^{-1},0)=\rho(1,\lambda^2b)$.
		\item $\tilde{\sigma}^2=Id. $
	\end{enumerate}
	Observe that $T$ is coherent with  the points from 3., 4. and 5. because $\rho$ is a homomorphism defined on $P$ (see \cref{bylambdasonhomomorfismos}). 
	By \cref{definiciondeA},  point 6. holds, therefore the only two families of relations left to be verified are that for every $b>0$,  
	$$ \begin{array}{rcl}\rho(b,0)&=&\sigma \rho(1,b^{-1})\sigma\rho(1,b)\sigma \rho(1,b^{-1})\\&=&\sigma \rho(1,-b^{-1})\sigma\rho(1,-b)\sigma \rho(1,-b^{-1}). \end{array}$$
	\begin{lem}\label{lemacomputos1}
		For $\epsilon=\pm 1$ and for every $b>0$,
		\begin{enumerate}
			\item	$1+K(\epsilon b)K(\epsilon b^{-1})+\langle Ac(\epsilon b),c(-\epsilon b^{-1})\rangle=0.$	
			\item	$	K(\epsilon b)c(\epsilon b^{-1})+\pi(\epsilon b^{-1})Ac(\epsilon b)=0.$
		\end{enumerate}
	\end{lem}
	\begin{proof}
		Indeed,  
		$$\begin{array}{rcl}
			1+K(\epsilon b)K(\epsilon b^{-1})+\langle Ac(\epsilon b),c(-\epsilon b^{-1})\rangle&=&\\	
			1+K(\epsilon)^2+K(\epsilon b)\langle c(-\epsilon b^{-1}),c(-\epsilon b^{-1})\rangle&=&\\
			1+K(\epsilon)^2+b^{-t}K(\epsilon)|c(-1)|^2&=&\\
			K(\epsilon)\Big(K(-\epsilon)+K(\epsilon)+|c(-1)|^2\Big)&=&0.
		\end{array}
		$$
		and 
		$$\begin{array}{rcl}
			K(\epsilon b)c(\epsilon b^{-1})+\pi(\epsilon b^{-1})Ac(\epsilon b)&=&\\
			K(\epsilon b)\big(c(\epsilon b^{-1})+\pi(\epsilon b^{-1})c(-\epsilon b^{-1})\big)&=&0. \end{array}$$
	\end{proof}
	
	\begin{lem}\label{relacionnomequedaclaro}
		If $b>0$ and $\epsilon=\pm1$, then
		$$ 
		\rho(b,0)=\tilde{\sigma}\rho(1,\epsilon b^{-1})\tilde{\sigma}\rho(1,\epsilon b)\tilde{\sigma}\rho(1,\epsilon b^{-1}).$$
	\end{lem}
	\begin{proof}
	The procedure will be to compare  the columns of the  canonical matrix representation of both side of the identities   with respect to the decomposition $\C\eta_1\oplus\C\eta_2\oplus L$. In fact, it will be shown that 
		$$\tilde{\sigma}\rho(b,0)\rho(1,-\epsilon b^{-1})\tilde{\sigma}=\rho(1,\epsilon b^{-1})\tilde{\sigma}\rho(1,\epsilon b).$$
		With a small abuse of notation, keep the notation above for the canonical linear representatives of each of the isometries. 
	Indeed, on one side, 
			$$\begin{array}{rl}
			\tilde{\sigma}\rho(b,0)\rho(1,-\epsilon b^{-1})\tilde{\sigma}(\eta_1)&=	\\
			\tilde{\sigma}\rho(b,0)\Big(K(-\epsilon b^{-1})\eta_1+\eta_2+c(-\epsilon b^{-1})\Big)&=\\
			\tilde{\sigma}\Big(b^tK(-\epsilon b^{-1})\eta_1+b^{-t}\eta_2+b^{-t}c(-\epsilon b)\Big)&=\\
			b^{-t}\eta_1+b^tK(-\epsilon b^{-1})\eta_2+b^{-t}K(-\epsilon b)c(\epsilon b^{-1})&=\\
			b^{-t}\eta_1+K(-\epsilon)\eta_2+K(-\epsilon)c(\epsilon b^{-1}),
		\end{array}
		$$		
		and on the other side, 
		$$
		\rho(1,\epsilon b^{-1})\tilde{\sigma}\rho(1,\epsilon b)(\eta_1)=
		K(\epsilon b^{-1})\eta_1+ \eta_2	+c(\epsilon b^{-1})	
		$$	
		Observe that 
		$$K(-\epsilon)\big(K(\epsilon b^{-1})\eta_1+ \eta_2	+c(\epsilon b^{-1})\big)=b^{-t}\eta_1+K(-\epsilon)\eta_2+K(-\epsilon)c(\epsilon b^{-1})$$
		Therefore as linear transformations, what has to be shown is that 
		$$\tilde{\sigma}\rho(b,0)\rho(1,-\epsilon b^{-1})\tilde{\sigma} =K(-\epsilon )\rho(1,\epsilon b^{-1})\tilde{\sigma}\rho(1,\epsilon b).$$
		
		For $\eta_2$, observe that, 
		$$
		\tilde{\sigma}\rho(b,0)\rho(1,-\epsilon b^{-1})\tilde{\sigma}(\eta_2)=b^t\eta_2
		$$
		and that 
		$$\begin{array}{rl}
			\rho(1,\epsilon b^{-1})\tilde{\sigma}\rho(1,\epsilon b)(\eta_2)&=\\
			\rho(1,\epsilon b^{-1})\tilde{\sigma}\Big(K(\epsilon b)\eta_1+\eta_2+c(\epsilon b)\Big)&=\\
			\rho(1,\epsilon b^{-1})\Big(\eta_1+K(\epsilon b)\eta_2+Ac(\epsilon b)\Big)&=\\
			\big(1+K(\epsilon b)K(\epsilon b^{-1})+\langle Ac(\epsilon b),c(-\epsilon b^{-1})\rangle\big)\eta_1+K(\epsilon b)\eta_2+&\\
			K(\epsilon b)c(\epsilon b^{-1})+\pi(\epsilon b^{-1})Ac(\epsilon b).
		\end{array}$$
		Therefore,  by \cref{lemacomputos1}, 
		$$\tilde{\sigma}\rho(b,0)\rho(1,-\epsilon b^{-1})\tilde{\sigma}(\eta_2) =K(-\epsilon)\rho(1,\epsilon b^{-1})\tilde{\sigma}\rho(1,\epsilon b)(\eta_2).$$
		
		And last, for every $d\in\R\setminus\{0\}$,  
		$$\begin{array}{rl}\tilde{\sigma}\rho(b,0)\rho(1,-\epsilon b^{-1})\tilde{\sigma}(c(d))&=\\
			K(d)\tilde{\sigma}\rho(b,0)\rho(1,-\epsilon b^{-1})c(-d^{-1})&=\\
			K(d)\tilde{\sigma}\rho(b,0)\Big(\langle c(-d^{-1}), c(\epsilon b^{-1})\rangle\eta_1+\pi(-\epsilon b^{-1})c(-d^{-1}) \Big)
			&=\\
			K(d)\tilde{\sigma}\Big(b^t\langle c(-d^{-1}), c(\epsilon b^{-1})\rangle\eta_1+\pi(b,0)\pi(-\epsilon b^{-1})c(-d^{-1}) \Big)&=\\
			K(d)\Big(b^t\langle c(-d^{-1}), c(\epsilon b^{-1})\rangle\eta_2+A\pi(b,0)\pi(-\epsilon b^{-1})c(-d^{-1}) \Big).
		\end{array}$$
		On the other hand, 
		$$\begin{array}{rl}
			\rho(1,\epsilon b^{-1})\tilde{\sigma}\rho(1,\epsilon b)(c(d))&=\\
			\rho(1,\epsilon b^{-1})\tilde{\sigma}\Big(\langle c(d), c(-\epsilon b^{})\rangle\eta_1+\pi(\epsilon b)c(d) \Big)&=\\
			\rho(1,\epsilon b^{-1})\Big(\langle c(d), c(-\epsilon b^{})\rangle\eta_2+A\pi(\epsilon b)c(d) \Big)	&=\\
			\Big(\langle c(d), c(-\epsilon b^{})\rangle K(\epsilon b^{-1})+\langle A\pi(\epsilon b)c(d),c(-\epsilon b^{-1})\rangle\Big) \eta_1 +&\\
			\langle c(d), c(-\epsilon b^{})\rangle\eta_2+	\langle c(d), c(-\epsilon b^{})\rangle c(\epsilon b^{-1}) +\pi(\epsilon b^{-1})A\pi(\epsilon b)c(d) .
		\end{array}$$
		Therefore the claim is that 
		$$\begin{array}{rcl}
			K(-\epsilon)\Big( 	\langle c(d), c(-\epsilon b^{})\rangle\eta_2+	\langle c(d), c(-\epsilon b^{})\rangle c(\epsilon b^{-1}) +\pi(\epsilon b^{-1})A\pi(\epsilon b)c(d)\Big)&=&\\
			K(d)\Big(b^t\langle c(-d^{-1}), c(\epsilon b^{-1})\rangle\eta_2+A\pi(b,0)\pi(-\epsilon b^{-1})c(-d^{-1}) \Big). 
		\end{array}$$
		Observe that 
		$$\begin{array}{rcl}
			K(\epsilon)K(d)b^t\langle c(-d^{-1}), c(\epsilon b^{-1})\rangle&=&\\
			\langle K(d)c(-d^{-1}), K(-\epsilon b)c(b^{-1})\rangle&=&\\
			\langle Ac(d),Ac(-\epsilon b)\rangle &=&\langle c(d),c(-\epsilon b)\rangle . 
		\end{array}$$
		Therefore 
		$$	K(-\epsilon)\langle c(d), c(-\epsilon b^{})\rangle=K(d)b^t\langle c(-d^{-1}), c(\epsilon b^{-1})\rangle.$$
		The only identity remaining  to show is that 
		$$\begin{array}{rcl}
			K(d)A\pi(b,0)\pi(-\epsilon b^{-1})c(-d^{-1})&=&\\
			K(-\epsilon)\Big(\langle	c(d), c(-\epsilon b^{})\rangle c(\epsilon b^{-1}) +\pi(\epsilon b^{-1})A\pi(\epsilon b)c(d) \Big).
		\end{array}$$
		Suppose $0\neq d\neq \epsilon b$. Notice that 
		$$\begin{array}{rcl}
			\pi(\epsilon b^{-1})A\pi(\epsilon b)c(d)&=&\\
			\pi(\epsilon b^{-1})A\big(c(\epsilon b+d)-c(\epsilon b)\big)&=&\\
			\pi(\epsilon b^{-1})\Big(K(\epsilon b+d)c(-(\epsilon b+d)^{-1})-K(\epsilon b)c(-\epsilon b^{-1}))
			\Big)&=&\\
			K( \epsilon b+d)\big(c(\frac{d}{\epsilon b(\epsilon b+d)})-c(\epsilon b^{-1})\big) +K(\epsilon b)c(\epsilon b^{-1})   &=&\\
			\big(K(\epsilon b)+K(d) + \langle c(d),c(-\epsilon b)\rangle \big)\big(c(\frac{d}{\epsilon b(\epsilon b+d)})-c(\epsilon b^{-1})\big)+
			K(\epsilon b)c(\epsilon b^{-1})&=&\\
			K(\epsilon b+d)c(\frac{d}{\epsilon b(\epsilon b+d)})-
			\big(K(d) + \langle c(d),c(-\epsilon b)\rangle \big)c(\epsilon b^{-1}).
		\end{array}$$
		Therefore if $R=K(-\epsilon)\pi(\epsilon b^{-1})A\pi(\epsilon b)c(d)$, 
		$$\begin{array}{rcl}
			K(-\epsilon)\Big(\langle	c(d), c(-\epsilon b^{})\rangle c(\epsilon b^{-1}) +\pi(\epsilon b^{-1})A\pi(\epsilon b)c(d) \Big)&=&\\
			K(-\epsilon)\langle	c(d), c(-\epsilon b^{})\rangle c(\epsilon b^{-1})	+R&=&\\
			K(-\epsilon)\Big(K(\epsilon b+d)c(\frac{d}{\epsilon b(\epsilon b+d)})-K(d)c(\epsilon b^{-1})\Big). 
		\end{array}$$
		
		On the other hand 
		$$\begin{array}{rcl}
			A\pi(b,0)\pi(-\epsilon b^{-1})c(-d^{-1})&=&\\
			A\pi(b,0)\big(c(-\frac{d\epsilon b}{\epsilon bd})-c(-\epsilon b^{-1})\big)&=&\\
			b^{-t}A\big(c(-\frac{\epsilon b(d\epsilon b)}{d})-c(-\epsilon b)\big)&=&\\
			b^{-t}\Big(K(-\frac{\epsilon b(d\epsilon b)}{d})c(\frac{d}{\epsilon b(d\epsilon b)})-K(-\epsilon b)c(\epsilon b^{-1})\Big)&=&\\
			K(-\frac{\epsilon (d\epsilon b)}{d})c(\frac{d}{\epsilon b(d\epsilon b)})-K(-\epsilon)c(\epsilon b^{-1}).
		\end{array}$$
		
		Therefore, what is left is to show that
		$$\begin{array}{rcl}
			K(d)\Big(	K(-\frac{\epsilon (d\epsilon b)}{d})c(\frac{d}{\epsilon b(d\epsilon b)})-K(-\epsilon)c(\epsilon b^{-1})\Big)&=&\\
			K(-\epsilon)\Big(K(\epsilon b+d)c(\frac{d}{\epsilon b(\epsilon b+d)})-K(d)c(\epsilon b^{-1})\Big),
		\end{array}$$
		which is equivalent to show that 
		$$	K(d)	K(\frac{-\epsilon d -b}{d})=	K(-\epsilon)K(\epsilon b+d).$$
		This is can be easily proved considering all the different cases. 
		
		Define  $$f_{\epsilon b}(d)=	K(d)A\pi(b,0)\pi(-\epsilon b^{-1})c(-d^{-1})$$ and 
		$$g_{\epsilon b }(d) =
		K(-\epsilon)\Big(\langle	c(d), c(-\epsilon b^{})\rangle c(\epsilon b^{-1}) +\pi(\epsilon b^{-1})A\pi(\epsilon b)c(d) \Big).$$
		
		Observe that for a given value $b_0\in\R\setminus\{0\}$, 
		the functions $f_{b_0}(d)$ and $g_{b_0}(d)$ are continuous on $d$ and such that $f_{b_0}(0)=0=g_{b_0}(0)$. It has been shown that for every $0\neq d\neq \epsilon b_0$, 
		$f_{b_0}(d) =g_{b_0}(d)$, therefore by continuity 
		$f_{b_0}=g_{b_0}$. 
		
		This concludes the proof for  the equalities,  $$\begin{array}{rcl}\rho(b,0)&=&\tilde{\sigma}\rho(1,b^{-1})\tilde{\sigma}\rho(1,b)\tilde{\sigma}\rho(1,b^{-1})\\&=&\tilde{\sigma}\rho(1,-b^{-1})\tilde{\sigma}\rho(1,-b)\tilde{\sigma}\rho(1,-b^{-1}). \end{array}
	$$
	\end{proof}
	The previous  lemma completes the argument that shows that $$SU(1,1)\xrightarrow{T}\iso(\hi_\C)_o$$ is a homomorphism. 
	The next theorem deals with the continuity and fixed point properties of $T$.  
	\begin{teo}\label{productodedosrepresentaciones}
		The map  $T$   induces an irreducible  (orbitally continuous)
		 representation $\rho$ of $\isocuno$ into $\text{Isom}(\hi_\C)_o$ with $\ell(\rho)=t$ and  $$\Arg(\rho)=\Arg(K_1(1)+K_2(1)). $$

	\end{teo}
	\begin{proof}  
		Observe that $T$ does not have fixed points in $\hi_\C\cup\partial\hi_C$ because $\sigma$ does not fix $\eta_1$ (see \cref{unicopuntofijo}). If $T$ preserves a geodesic, then it permutes the two limits of it, but this is a contradiction because   every  homomorphism $SU(1,1)\xrightarrow{}\Z_2$ is constant. 
		
		Let $SU(1,1)\xrightarrow{\pi}\text{Isom}(\mathbf{H}^1_\C)_o$ be the projectivization map. The group $\pi(P)$ is closed in $\text{Isom}(\mathbf{H}^1_\C)_o$  and, by \cref{transtividad}, there is a decomposition $$\text{Isom}(\mathbf{H}^1_\C)_o=\pi(PsP)\sqcup \pi(P).$$  Therefore  $\pi(PsPs) $ is an open neighborhood of $Id\in \text{Isom}(\mathbf{H}^1_\C)_o$. Thus, it is enough to show that, if $(g_j)$ is a sequence in $ \pi(Ps P)$  such that $(g_j)\to \pi(s)$, then for every $x\in\hi_C$, $$\rho(g_j)x\to T(s)x=\tilde{\sigma} x.$$
		Observe that every element of $P s P$ can be written as 
		$$g(\lambda,b)sg(1,d)=
		\begin{pmatrix}
			-b&i(\lambda-bd)\\
			i\lambda^{-1}&-\lambda^{-1}d
		\end{pmatrix}.$$
		If $$g_j=\pi(g(\lambda_j,b_j) sg(1,d_j)),$$
		then $b_j\to0$, $\lambda_j\to1$ and $d_j\to0$. 
		Therefore,  for every $x\in\hi_C$, $\rho(\lambda_j,b_j)x\to x$ and  $\rho (1,d_j)x\to x$, hence with a triangle inequality argument it is possible to conclude  that $g_jx\to \sigma x$.  
		
		The irreducible part of $\rho$   has to contain the axis (and its limits) preserved by the maps $\rho(\lambda,0)$, therefore  $\rho$ is irreducible by construction.  
	\end{proof}
\subsection{A new family of representations }
With the results of the previous subsection a continuum of non-equivalent representations will be constructed.  

Given an irreducible  representation $\rho$ denote $K(1)=K(1)_\rho$.
If $p,q\in\R>0$ and $\rho,\tau:\isocuno\rightarrow\iso(\mathbf{H}_\C^\infty)$ are two irreducible representations such that $\ell(\rho)=\ell(\tau)=t$, let $\rho_p$ and $\tau_q$ be two irreducible representations, equivalent to $\rho$ and $\tau$ respectively, such that $|K(1)_{\rho_p}|=p$ and $|K(1)_{\tau_q}|=q$. 
Observe that with the procedure describe in \cref{productodedosrepresentaciones} it is possible to obtain an irreducible representation $\omega$ such that $\ell(\varphi)=t$ and 
$$\Arg(\varphi)=\Arg\left(\frac{pK_\rho(1)}{|K_\rho(1)|}+\frac{qK_\tau(1)}{|K_\tau(1)|}\right).$$ 
Therefore for every $$s\in[\min\{\Arg(\rho),\Arg(\tau)\}, \max\{\Arg(\rho),\Arg(\tau)\}]$$ there is an irreducible  representation $\phi$ such that  $\ell(\phi)=t$ and $\Arg(\phi)=s$. 

Given $u\in[0,1]$, denote $\rho\underset{u}{\wedge}\tau$ the irreducible representation such that $\ell(\rho\underset{u}{\wedge}\tau)=t$ and $$\Arg(\rho\underset{u}{\wedge}\tau)=(1-u)\Arg(\rho)+u\Arg(\tau).$$ 
This representation will be called a  $\textit{horospherical combination} $ of $\rho$ and $\tau$.  

The representation $\rho\underset{u}{\wedge}\tau$ is well defined in the following sense. If $\ell(\rho)=\ell(\tau)$ and $\rho'$ and $\tau'$ are equivalent to $\rho$ and $\tau$ respectively, then $\rho\underset{u}{\wedge}\tau$ is equivalent to $\rho'\underset{u}{\wedge}\tau'$ (see \cref{clasificacionconK}). 
	
	Although in the definition of the horospherical combination,  for simplicity,  the representations were supposed  acting on  the same hyperbolic space,  nothing prevents to define the horospherical combination of two  irreducible representations  with one  possibly having   finite-dimensional target.  		
This could be  the case, by Mostow-Karpelevich theorem or in particular by \cref{linealmenteindependiente}, if $t=1$. 

Using the  families constructed in \cite{monod2018notes} and \cite{monod2014exotic} and the horospherical combination a new family of non-equivalent representations is built.  

 	Recall that for every $0<t<2$ on one hand ,  up to a conjugation, there exists a unique   irreducible  representation $$\iso(\mathbf{H}^1_\C)_o\xrightarrow{\rho_t}\iso(\hi_\C)_o$$ such that  $\ell(\rho_t)=t$ and that preserves a real hyperbolic space. These representations are such that $\Arg({\rho_t})=0$ (see \cref{Deltaesceroparacomplexificaciones}, the comments before it and \cref{clasificacionconK}). 	
	On the other hand, for every $0<t<1$ there exists  an  irreducible representation $$\isocuno\xrightarrow{\tau_t}\iso(\hi_\C)_o$$ such that
 $\Arg(\tau_t)=\frac{t\pi}{2}$ and $\ell(\tau_t)=t$ (see \cref{cartandelapotencidelatautologica} and the comments before it).

	\begin{teo}
	If $0<t<1$ and   $r\in[0,{t\pi}/{2}]$ or if $t=1$ and $r\in[0,{\pi}/{2})$,  there exists a unique, up to a  conjugation,   irreducible representation  $\rho_{t,r}$  such that  $\Arg(\rho_{t,r})=r$ and $\ell(\rho_{t,r})=t$.
	\end{teo}
	\begin{proof}
	
		For $t<1$, consider the family of irreducible representations  $\rho_t\underset{u}{\wedge} \tau_t$. 

For $t=1$, let $id$ be the identity map $\iso(\mathcal{H}_\C^1)_o\rightarrow\iso(\mathcal{H}_\C^1)_o.$
Observe that for every $u\in[0,1)$,  by construction the representation $\rho_t\underset{u}{\wedge} id$ is irreducible and the target is an infinite-dimensional complex hyperbolic space. \end{proof}
By \cref{cartandimensioninfinita}, the  representations listed in the previous theorem are representatives of all the irreducible representations of $\isocuno$ into $\iso(\mathbf{H}_\C^\infty)_o$ with displacement 1. 
\section*{Acknowledgments}
I would like to express my deep gratitude to Nicolas Monod for proposing the question and for the many interesting and very revealing discussions occurred in the spring of 2022 without which this work would not have been possible.

	\bibliographystyle{plain}
	\bibliography{biblio}

\begin{thebibliography}{10}

\bibitem{propertyTbekka}
Bachir Bekka, Pierre de~la Harpe, and Alain Valette.
\newblock {\em Kazhdan's property ({T})}, volume~11 of {\em New Mathematical
  Monographs}.
\newblock Cambridge University Press, Cambridge, 2008.

\bibitem{boubelkarpelevichmost}
Charles Boubel and Abdelghani Zeghib.
\newblock Isometric actions of {L}ie subgroups of the {M}oebius group.
\newblock {\em Nonlinearity}, 17(5):1677--1688, 2004.

\bibitem{bridson2013metric}
Martin~R. Bridson and Andr\'{e} Haefliger.
\newblock {\em Metric spaces of non-positive curvature}, volume 319 of {\em
  Grundlehren der mathematischen Wissenschaften [Fundamental Principles of
  Mathematical Sciences]}.
\newblock Springer-Verlag, Berlin, 1999.

\bibitem{burgeriozziboundedcohomology}
Marc Burger and Alessandra Iozzi.
\newblock Bounded cohomology and totally real subspaces in complex hyperbolic
  geometry.
\newblock {\em Ergodic Theory and Dynamical Systems}, 32(2):467--478, 2012.

\bibitem{burger2005equivariant}
Marc Burger, Alessandra Iozzi, and Nicolas Monod.
\newblock Equivariant embeddings of trees into hyperbolic spaces.
\newblock {\em International Mathematics Research Notices}, (22):1331--1369,
  2005.

\bibitem{capracelytchak}
Pierre-Emmanuel Caprace and Alexander Lytchak.
\newblock At infinity of finite-dimensional {CAT}(0) spaces.
\newblock {\em Mathematische Annalen}, 346(1):1--21, 2010.

\bibitem{das2017geometry}
Tushar Das, David Simmons, and Mariusz Urba\'{n}ski.
\newblock {\em Geometry and dynamics in {G}romov hyperbolic metric spaces},
  volume 218 of {\em Mathematical Surveys and Monographs}.
\newblock American Mathematical Society, Providence, RI, 2017.
\newblock With an emphasis on non-proper settings.

\bibitem{pydelzant}
Thomas Delzant and Pierre Py.
\newblock K\"{a}hler groups, real hyperbolic spaces and the {C}remona group.
\newblock {\em Compositio Mathematica}, 148(1):153--184, 2012.

\bibitem{duchesne2021}
Bruno Duchesne, Jean Lécureux, and Maria~Beatrice Pozzetti.
\newblock Boundary maps and maximal representations on infinite-dimensional
  hermitian symmetric spaces.
\newblock {\em Ergodic Theory and Dynamical Systems}, page 1–50, 2021.

\bibitem{complexhyperbolic}
William~M. Goldman.
\newblock {\em Complex hyperbolic geometry}.
\newblock Oxford Mathematical Monographs. The Clarendon Press, Oxford
  University Press, New York, 1999.
\newblock Oxford Science Publications.

\bibitem{sl2rlang}
Serge Lang.
\newblock {\em {${\rm SL}_2({\bf R})$}}, volume 105 of {\em Graduate Texts in
  Mathematics}.
\newblock Springer-Verlag, New York, 1985.
\newblock Reprint of the 1975 edition.

\bibitem{monod2018notes}
Nicolas Monod.
\newblock Notes on functions of hyperbolic type.
\newblock {\em Bulletin of the Belgian Mathematical Society. Simon Stevin},
  27(2):167--202, 2020.

\bibitem{monod2014exotic}
Nicolas Monod and Pierre Py.
\newblock An exotic deformation of the hyperbolic space.
\newblock {\em American Journal of Mathematics}, 136(5):1249--1299, 2014.

\bibitem{monodpyselfrepresentations}
Nicolas Monod and Pierre Py.
\newblock Self-representations of the {M}\"{o}bius group.
\newblock {\em Annales Henri Lebesgue}, 2:259--280, 2019.

\bibitem{Mostowmostowkarpelevich}
G.~D. Mostow.
\newblock Some new decomposition theorems for semi-simple groups.
\newblock {\em Memoirs of the American Mathematical Society}, 14:31--54, 1955.

\bibitem{translationlengthsparabolic}
Yunhui Wu.
\newblock Translation lengths of parabolic isometries of {${\rm CAT}(0)$}
  spaces and their applications.
\newblock {\em Journal of Geometric Analysis}, 28(1):375--392, 2018.

\end{thebibliography}
	
\end{document}